\documentclass[a4paper,10pt]{article}
\usepackage[colorlinks,bookmarksopen,bookmarksnumbered,colorlinks = false]{hyperref}
\usepackage{caption}
\usepackage{graphicx}
\usepackage[ruled,vlined]{algorithm2e}
\usepackage{array}
\usepackage{verbatim}
\usepackage{multirow}
\usepackage{pgfplots}
\pgfplotsset{compat=1.8}
\usepackage{amssymb}
\usepackage{amsmath}
\usepackage{textcomp}
\usepackage{algorithmic}
\usepackage{fancyhdr}

\fancypagestyle{alim}{\fancyhf{}\fancyfoot[L]{\fontsize{9}{11} \selectfont  $^*$email: \texttt{T.J.Deveney@bath.ac.uk}\\ $^\dagger$email: \texttt{E.Mueller@bath.ac.uk}\\$^\ddagger$email: \texttt{T.Shardlow@bath.ac.uk}}}
\usepackage{bm}
\usepackage{float}
\newcommand\numberthis{\addtocounter{equation}{1}\tag{\theequation}}
\RequirePackage[capitalize,nameinlink]{cleveref}[0.19]
\usepackage{amsmath,a4,amssymb,textcomp,latexsym,csquotes, xcolor,mathtools,hyperref,microtype}
\crefname{assumption}{Assumption}{Assumptions}

\DeclarePairedDelimiter{\Vpair}{\|}{\|} 
\newcommand\norm[1]{\Vpair*{{#1}}}

\newcommand{\tmop}[1]{\ensuremath{\operatorname{#1}}}

\usepackage[latin1]{inputenc}
\usepackage{tikz}
\usepackage{multirow}
\usetikzlibrary{trees}

\usepackage{amssymb}
\usepackage[colorlinks,bookmarksopen,bookmarksnumbered]{hyperref}
\usepackage{mathtools,xparse}
\usepackage{caption}
\title{Deep surrogate accelerated delayed-acceptance HMC for Bayesian inference of spatio-temporal heat fluxes in rotating disc systems}
\author{\fontsize{13}{11} \selectfont Teo Deveney$^{*}$, Eike Mueller$^{\dagger}$, Tony Shardlow$^{\ddagger}$ \\ \fontsize{10}{11} \selectfont Department of Mathematical Sciences, University of Bath, Bath, UK, BA2 7AY}

\renewcommand{\theequation}{\arabic{equation}}

\newtheorem{theorem}{Theorem}[section]
\newtheorem{lemma}[theorem]{Lemma}

\newtheorem{assumption}[theorem]{Assumption}
\newtheorem{corollary}[theorem]{Corollary}

\newenvironment{proof}{{\bf Proof:}}{\hfill\rule{2mm}{2mm}}
\date{}

\usepackage[a4paper, top=1.5in, total={6in, 9in}]{geometry}

\begin{document}

\numberwithin{equation}{section}

\maketitle
\thispagestyle{alim}
\begin{abstract}
\noindent
We introduce a deep learning accelerated methodology to solve PDE-based Bayesian inverse problems with guaranteed accuracy. This is motivated by solving the ill-posed problem of inferring a spatio-temporal heat-flux parameter known as the Biot number in a PDE model given temperature data, however the methodology is generalisable to other settings. To achieve accelerated Bayesian inference we develop a novel training scheme that uses data to adaptively train a neural-network surrogate simulating the parametric forward model. By simultaneously identifying an approximate posterior distribution over the Biot number, and weighting a physics-informed training loss according to this, our approach approximates forward and inverse solution together without any need for external solves. Using a random Chebyshev series, we outline how to approximate a Gaussian process prior, and using the surrogate we apply Hamiltonian Monte Carlo (HMC) to sample from the posterior distribution. We derive convergence of the surrogate posterior to the true posterior distribution in the Hellinger metric as our adaptive loss approaches zero. Additionally, we describe how this surrogate-accelerated HMC approach can be combined with traditional PDE solvers in a delayed-acceptance scheme to a-priori control the posterior accuracy. This overcomes a major limitation of deep learning-based surrogate approaches, which do not achieve guaranteed accuracy a-priori due to their non-convex training. Biot number calculations are involved in turbo-machinery design, which is safety critical and highly regulated, therefore it is important that our results have such mathematical guarantees. Our approach achieves fast mixing in high-dimensional parameter spaces, whilst retaining the convergence guarantees of a traditional PDE solver, and without the burden of evaluating this solver for proposals that are likely to be rejected. A range of numerical results are given using real and simulated data that compare adaptive and general training regimes and various gradient-based Markov chain Monte Carlo (MCMC) sampling methods.
\end{abstract}

\newpage

\section{Introduction}
\let\thefootnote\relax\footnotetext{To appear in SIAM/ASA journal on uncertainty quantification}

The disc temperature distribution in compressor cavities is a fundamental quantity of interest for aerospace engineers due to its effect on material expansion. In order to improve engine design, engineers are interested in simulating the temperature evolution over time, leading to an urgent requirement for accurate physical models of heat transfer inside the engine cavity. Appropriate parameterisations of these models in a transient setting are currently not well understood, and therefore experimental data can be extremely valuable in aiding our knowledge of the parameters. One very important parameter in such models is the Biot number, a function which can vary over space and time, that dictates the relative effects of convection and conduction on heat transfer.
\\ \indent
Previous work has been carried out to infer the Biot number from temperature measurements in a stationary setting. In \cite{badcurves}, polynomial curves were fit to data using least squares. This approach, while yielding a good fit, lacks appropriate regularisation which often leads to physically implausible inferences with large oscillations unless restricted to very low degree polynomials. This effect is highlighted in \cite{biot}, where a Bayesian regularisation method is instead proposed based on maximum a posteriori (MAP) estimation over a spatial discretisation of the Biot number. This approach is shown to yield physical results and a local estimate of the uncertainty is achieved using a Laplace approximation based on the Hessian of the log-posterior. Using this Laplace approach, the full posterior distribution is not returned, meaning the uncertainty estimate may be unreliable. Furthermore, the increase in number of degrees of freedom of the Biot number coupled with the higher complexity of the PDE solve in the spatio-temporal setting pursued in this work, ensure that the extension of the approach of \cite{biot} to this case is computationally intractable. This is because it is reliant on numerically calculating the gradient of the posterior with respect to each degree of freedom of the discretised Biot number. More efficient approaches to achieve this are possible by using the adjoint PDE to solve the PDE-constrained optimisation problem, though a large number of numerical solves is still required in this case, and the local estimate of the associated uncertainty returned is insufficient as our numerical results reveal.
\\ \indent
In this work, we overcome the limitations imposed by traditional numerical schemes in the spatio-temporal setting, by using deep learning to develop a Bayesian methodology capable of approximating the full posterior distribution of the Biot number. Our approach is to represent the parametric forward map by a neural network, thus greatly accelerating the simulation and differentiation of the PDE model. To attain this map efficiently, we design an adaptive training scheme, based on minimising the squared PDE-residual over a measure that approximates the true posterior over the parameters. The relationship between our adaptive training loss function value and the accuracy of the posterior approximation induced by this surrogate is analysed in the Hellinger metric. We demonstrate that this restriction to the posterior measure vastly reduces the training time (from over 4 hours to around 15 minutes in our experiments), whilst also improving the approximation accuracy (by a factor of over 20 in $L_1$-error), when compared to a general parametric approximation over a wider parameter space. Given our approximate forward map, we use the Hamiltonian Monte Carlo (HMC)  sampling scheme \cite{betancourt2015hamiltonian} to generate proposal samples from the posterior distribution. Using this method, we are able to perform a full Bayesian analysis of the posterior distribution in minutes. Our results on simulated data show that a fully Bayesian approach is justified, as it provides a more accurate quantification of uncertainty than the Laplace approximation, which gives overly confident results for this problem. 
\\ \indent
PDE surrogates based on the approximation of parametric solutions by neural networks have been applied with impressive results previously \cite{yan21,zhu19,teo_deepsurrogate}. However, due to the non-convex nature of the training procedure, these approaches suffer from an inability to guarantee the accuracy of the approximations that are attained. To overcome this, we additionally propose delayed-acceptance as part of our MCMC scheme \cite{fox_da}. In this delayed-acceptance HMC scheme, proposals which pass the initial surrogate-based Metropolis acceptance criterion are passed to a secondary acceptance criterion dependent on a finite-difference (FD) solver. The secondary criterion is chosen such that the stationary distribution of the Markov chain satisfies detailed balance according to the likelihood induced by the FD solver. This approach, while slower than relying solely on the deep learning surrogate, ensures that the computation time dedicated to FD is being used optimally, since FD is only executed for proposals that have passed initial acceptance criterion and therefore have a high probability of acceptance, and successive proposals are decorrelated by the Hamiltonian proposal distribution. Furthermore, the FD solver is a well-studied space-time discretisation of the PDE solution with a rich convergence theory and quantifiable error \cite{Thomas1999-kz}. As a result, we obtain a posterior sample that has the accuracy and convergence guarantees associated with a FD solver, but at a significantly lower computational cost than is possible by using FD in a typical Metropolis--Hastings sampler. 
\\ \indent
The remainder of this work proceeds as follows. In Section \ref{specify}, we fully outline the Bayesian inverse problem for the Biot number that we consider throughout this work. Section \ref{sec:method} describes our methodology in various parts, beginning with how to specify a Gaussian process prior for the Biot number within a deep-learning surrogate, proceeded by a description of our adaptive training scheme, then the deep-surrogate-accelerated delayed-acceptance HMC sampling. In Section \ref{sec:results}, we apply this methodology to simulated and real experimental data. Our experiments compare the efficiency and accuracy of our adaptively trained surrogate to a surrogate trained over a more general set of parameters, and quantifies the statistical accuracy of different sampling schemes in terms of the effective sample size (ESS) \cite{gey92} obtained. These experiments are carried out with and without the delayed-acceptance step, and the accuracy of the various sampling schemes and ``surrogate-only'' approaches (without delayed-acceptance) are visualised through posterior density plots. 

\section{Problem specification} \label{specify}
We consider heat transfer in rotating disc systems. To model the disc, we make an axisymmetric assumption in one spatial dimension representing radial location, and consider the evolution of the heat profile over time. After non-dimensionalisation, an appropriate PDE for the temperature of the disc is the transient \textit{fin equation}
\begin{align}
	c_0\frac{\partial u}{\partial t}(t,x) = c_1\frac{\partial^2u}{\partial x^2}(t,x) + \frac{c_2}{x} \frac{\partial u}{\partial x}(t,x) - Bi(t,x) u(t,x),\qquad  t\in[0,T], x\in [a,b], \label{fineq}
\end{align}
for $0<a<b$ and known fixed parameters $c_0,c_1,c_2>0$.
To this equation, we prescribe Dirichlet boundary conditions
\begin{align*}
	u(t,a) &= u_a(t), \qquad t\in [0,T],\\
	u(t,b) &= u_b(t), \qquad t\in [0,T], \label{BC} \numberthis 
\end{align*}
and an initial condition
\begin{align}
	u(0,x) = u_0(x),\qquad x\in [a,b]. \label{IC}
\end{align}
Given data $\left\{(\hat{t}_n,\hat{x}_n,\hat{z}_n): n = 1,2,\dots,N\right\}$, our goal is to perform Bayesian uncertainty quantification for the spatio-temporal parameter $Bi(t,x)$ in \eqref{fineq} known as the Biot number. Here $(\hat{t}_n,\hat{x}_n)$ represent space-time coordinates and $(\hat{z}_n)$ represent temperature measurements corresponding to these coordinates. 
\\ \indent
We assume the data is related to the PDE through the statistical model 
\begin{align}
    \qquad \qquad \qquad \qquad \hat{z}_n = u(\hat{t}_n,\hat{x}_n) + \epsilon_n, \qquad \qquad n=1,\dots,N, \label{statmodel}
\end{align}
where $u(\cdot,\cdot)$ is the solution to (\ref{fineq}, \ref{BC}, \ref{IC}) and the noise terms $\epsilon_n \sim N(0,\sigma_\epsilon^2)$ are i.i.d. Gaussian random variables with unknown standard deviation $\sigma_\epsilon$. Under this model, the likelihood function for given parameters $Bi(t,x),\sigma_\epsilon$ is given by the conditional probability density function of the observed data  
\begin{align}
    p(\hat{z}|\hat{t},\hat{x},Bi,\sigma_\epsilon) = \frac{1}{{\left(2\pi\sigma_\epsilon^2\right)^{N/2} }}\text{exp}\left( - \frac{1}{2\sigma_\epsilon ^2 }\sum_{n=1}^N( {\hat{z}_n - u(\hat{t}_n,\hat{x}_n) })^2  \right). \label{likelihood}
\end{align}
To perform Bayesian inference, we must define a prior distribution $p(Bi,\sigma_\epsilon)$ over the unknown parameters $Bi(t,x)$ and $\sigma_\epsilon$. The choice of prior distribution is important since it can be used as a regulariser to ensure that we recover physically meaningful results. Here we will assume that $Bi(t,x)$ and $\sigma_\epsilon$ are independent in the prior distribution, so that $p(Bi,\sigma_\epsilon) = p_B(Bi)p_\sigma(\sigma_\epsilon)$ and we define $p_\sigma(\sigma_\epsilon)$ to be a Gamma distribution. For $p_B(Bi)$, we require the prior distribution to be defined over the space of 2-dimensional functions, and for this we choose the popular option of assigning a Gaussian process prior
\begin{align}
    Bi \sim \mathcal{GP}\left(\mu(\cdot),K(\cdot,\cdot)\right).
\end{align}
Here the mean function $\mu(\cdot)$ and covariance kernel $K(\cdot,\cdot)$ are manually specified such that they represent our prior beliefs about the behaviour of the function. In this work, we discretise the Gaussian process by a truncated random-series expansion and the infinite-dimensional prior $p_B(Bi)$ is approximated by a density over finitely many discretisation parameters.
\\ \indent
Having decided the prior distribution and likelihood \eqref{likelihood}, an application of Bayes rule gives us the posterior distribution
\begin{align}
    p(Bi,\sigma_\epsilon|\hat{t},\hat{x},\hat{z}) = \frac{p(Bi,\sigma_\epsilon)p(\hat{z}|\hat{t},\hat{x},Bi,\sigma_\epsilon)}{p(\hat{z}|\hat{t},\hat{x})} \propto p(Bi,\sigma_\epsilon)p(\hat{z}|\hat{t},\hat{x},Bi,\sigma_\epsilon). \label{bayesformula}
\end{align}
Approximating the posterior distribution is the aim of the Bayesian inverse problem. This measure represents the full conditional distribution of all unknown parameters given the observed data and prior information. Analysis of this distribution provides a robust and natural form of uncertainty quantification over the parameters, whilst also providing a complete picture of any correlations, skews, or heavy tails present in the distribution that might be important when assessing the results in practice. 

\section{Methodology}  \label{sec:method}
Our approach broadly consists of 3 stages:
\begin{enumerate}
    \item Represent $Bi(t,x)$ by a series expansion and compute the prior distribution over the coefficients such that the series approximates a Gaussian process prior.
    \item Train a deep learning surrogate model to approximate the parametric forward problem on an appropriate measure over the series coefficients.
    \item Using the computed prior distribution, efficiently sample from the posterior distribution induced by the FD solution by combining surrogate accelerated HMC proposals with a FD solver in a delayed-acceptance MCMC scheme.  
\end{enumerate}
Here we describe each of these steps in detail, outlining various options and their relative benefits in practice.

\subsection{Prior approximation}
\label{sec:prior}
We represent the functional parameter $Bi(t,x)$ through the Chebyshev expansion
\begin{align}
    \hat{Bi}(t,x) = \sum_{i=1}^{M} \alpha_iT_{i}(t,x). \label{chebser}
\end{align}
In this expression the $T_{i}(t,x)$ are two dimensional Chebyshev polynomial basis functions. The maximum degree of these Chebyshev basis functions is D, therefore the number of basis terms is $M=\frac{1}{2}(D+1)(D+2)$, where for each $k,l\in \{0,1,\dots,D\}$ such that $k+l\leq D$ there is a unique $i\in \{1,\dots,M\}$ such that $T_{i}(t,x) = T_k(t)T_l(x)$. $T_k(t), T_l(x)$ are the shifted Chebyshev polynomials of the first kind onto the domains $t\in [0,T]$ and $x\in [a,b]$ respectively. The coefficients $\boldsymbol{\alpha}\in  \mathbb{R}^M$ are parameters that we wish to infer using data, and will be given as inputs to the deep learning surrogate model. Given this representation of $Bi(t,x)$ as a linear basis expansion, we seek to define a meaningful prior distribution over the coefficients of this expansion. The Gaussian process class of distributions is a natural target, as it allows us to define prior distributions over functions in a way that is easily interpretable through the specification of the mean function and covariance kernel. We ensure that $\hat{Bi}(t,x)$ is itself a Gaussian process by assigning a multivariate normal prior over the coefficients
\begin{align}
    \boldsymbol{\alpha} \sim MVN(\boldsymbol{m}, \Sigma), \label{coefprior}
\end{align}
where $\boldsymbol{m}\in \mathbb{R}^M$ and $\Sigma\in \mathbb{R}^{M\times M}$ are the mean vector and covariance matrix. 
\\ \indent
The mean function of $\hat{Bi}(t_j,x_j)$ is straightforward to calculate as
\begin{align}
    \hat{\mu}(t,x) = \sum_{i=1}^M m_iT_{i}(t,x), \label{mean}
\end{align}
and the covariance function is 
\begin{align}
    \hat{K}([t,x],[t',x']) &= \sum_{i,j=1}^M T_{i}(t,x)\Sigma_{i,j} T_{j}(t',x'). \label{covariance} 
\end{align}
We can then approximate a Gaussian process using \eqref{chebser} by approximating the desired deterministic mean and covariance using (\ref{mean}, \ref{covariance}). More explicitly, suppose that we wish to approximate the Gaussian process $Bi(t,x) \sim \mathcal{GP}(\mu(t,x),K([t,x],[t',x']))$; we can compute the mean vector $(m_1,\dots,m_M)$ such that
\begin{align}
    \sum_{i=1}^M m_iT_{i}(t,x) \approx \mu(t,x), \label{meanss}
\end{align}
and the covariance matrix entries $(\Sigma_{1,1},\Sigma_{1,2},\dots,\Sigma_{M,M})$ such that
\begin{align}
    \sum_{i,j=1}^MT_{i}(t,x)\Sigma_{i,j}T_{j}(t',x') \approx K([t,x],[t',x']).  \label{covss}
\end{align}
The computation of $\boldsymbol{m}$ and $\Sigma$ are function approximation problems that are achieved by interpolation on Chebyshev nodes using the barycentric formula. A detailed treatment of the theory of approximation using Chebyshev interpolants is given in \cite{tref15}. Notably this approximation procedure is computationally stable and efficient. Both the mean and covariance converge in the maximum norm as the degree N of the Chebyshev polynomial increases. The rate of convergence is $O(N^{-k})$ for functions with $k^{th}$ derivative of bounded variation, and geometric for analytic functions. Another approach that we could apply is the Karhunen--Lo\`eve (KL) expansion \cite{Lord2014-co}. In the KL expansion the basis functions and the distribution of the coefficients are solutions to an eigenvalue problem, and the corresponding approximation is the best linear approximation to the Gaussian field in terms of the mean-squared error. We note that while the KL expansion defines a more accurate approximation to the GP, it is inconvenient in practice as it requires us to train a new surrogate with different basis functions each time we change the prior distribution. By instead taking the approach described above, we require only one surrogate based on Chebyshev polynomials, and if we wish to change the prior we need only recompute mean vector and covariance matrix of the random vector $\boldsymbol{\alpha}$. 
\\ \indent 
The resultant distribution of $\hat{Bi}(t,x)$ is a Gaussian process with mean function and covariance kernel approximately equal to those of the original Gaussian process $Bi(t,x)$, however it is expressed entirely by a multivariate normal distribution over the coefficients $p_\alpha(\boldsymbol{\alpha})$, with mean and covariance according to $\eqref{coefprior}$. We can subsequently use this as the prior distribution for our method. The corresponding posterior distribution is therefore
\begin{align}
    p(\boldsymbol{\alpha},\sigma_\epsilon|\hat{t},\hat{x},\hat{z}) = \frac{p(\boldsymbol{\alpha},\sigma_\epsilon)p(\hat{z}|\hat{t},\hat{x},\boldsymbol{\alpha},\sigma_\epsilon)}{p(\hat{z}|\hat{t},\hat{x})} \propto p(\boldsymbol{\alpha},\sigma_\epsilon)p(\hat{z}|\hat{t},\hat{x},\boldsymbol{\alpha},\sigma_\epsilon), \label{bayesformula_alph}
\end{align}
where $p(\boldsymbol{\alpha},\sigma_\epsilon) = p_\alpha(\boldsymbol{\alpha})p_\sigma(\sigma_\epsilon)$, and  $p(\hat{z}|\hat{t},\hat{x},\boldsymbol{\alpha},\sigma_\epsilon)$ is the likelihood function 
\begin{align}
    p(\hat{z}|\hat{t},\hat{x},\boldsymbol{\alpha},\sigma_\epsilon) = \frac{1}{{\left(2\pi\sigma_\epsilon^2\right)^{N/2} }}\text{exp}\left( - \frac{1}{2\sigma_\epsilon ^2 }\sum_{n=1}^N( {\hat{z}_n - u(\hat{t}_n,\hat{x}_n,\boldsymbol{\alpha}) })^2  \right). \label{likelihood_alph}
\end{align}
Here $u(\hat{t},\hat{x},\boldsymbol{\alpha})$ is the parametric solution to the PDE over the coefficients $\boldsymbol{\alpha}$, which will be approximated by a deep surrogate model in our method. 

\subsection{Deep surrogate approximation}
\label{sec:surrogate_approx}
Now we turn to approximating the parametric solution to the fin equation \eqref{fineq} over the parameters $\boldsymbol{\alpha}$. Let us define the interior PDE domain $\Omega = [0,T]\times[a,b]$, and associate to it a positive measure $\pi$. Similarly define the boundary domain $\partial \Omega = [0,T]\times\{a,b\}\cup\{0\}\times[a,b]$ with positive measure $\pi^b$, and note that the boundary and initial conditions (\ref{BC}, \ref{IC}) can be combined into a single Dirichlet condition
\begin{align}
    u(t,x) = u_{BC}(t,x), \qquad (t,x) \in \partial \Omega.
\end{align}
To construct the deep surrogate model, we approximate the parametric solution to \eqref{fineq} using a neural network $\hat{u}(t,x,\boldsymbol{\alpha})\colon \Omega \times \mathbb{R}^{M}\to\mathbb{R}$. 
We associate a positive measure $\pi^\alpha$ on $\mathbb{R}^M$ to these coefficients, and train the network using stochastic gradient descent to minimise the loss function
\begin{align*}
\text{Loss}&=	\int_{\mathbb{R}^M} F(\boldsymbol \alpha) \,d\pi^\alpha(\boldsymbol\alpha),\\
    F(\boldsymbol \alpha)&=    \nu_1\|\mathcal{L}\hat{u}(\cdot,\cdot,\boldsymbol{\alpha})-b(\cdot,\cdot)\|^2_{L_2(\Omega,\pi)} \\&\quad+ \nu_2\|\hat{u}(\cdot,\cdot,\boldsymbol{\alpha}) - u_{BC}(\cdot,\cdot,\boldsymbol \alpha)\|^2_{L_2(\partial \Omega,\pi^b)}, \label{loss} \numberthis 
\end{align*}
for some parameters $\nu_1,\nu_2>0$.
Here $b=0$, and $\mathcal{L}\hat{u}$ is the differential operator of the fin equation applied to the surrogate
\begin{align}
    \mathcal{L}\hat{u}(t,x,\boldsymbol{\alpha}) = \frac{\partial^2\hat{u}}{\partial x^2}(t,x,\boldsymbol{\alpha}) + \frac{1}{x} \frac{\partial \hat{u}}{\partial x}(t,x,\boldsymbol{\alpha}) - \hat{Bi}(t,x) \hat{u}(t,x,\boldsymbol{\alpha}) - \frac{\partial \hat{u}}{\partial t}(t,x,\boldsymbol{\alpha}).
\end{align}
We use the subscript $L_2(C,\mu)$ notation to denote the $L_2$-norm over domain $C$ with respect to the measure $\mu$ 
\begin{align}
    \|h\|^2_{L_2(C,\mu)} = \int_C|h(\boldsymbol{x})|^2d\mu(\boldsymbol{x}). \label{norm}
\end{align}
The integrals in \eqref{loss} are intractable, so in practice this minimisation is implemented by drawing randomised collocation points $(t,x,\boldsymbol{\alpha})$ from the interior measure $\pi\otimes\pi^\alpha$ and boundary points $(t_b,x_b,\boldsymbol{\alpha_b})$ from $\pi^b\otimes\pi^\alpha$, then minimising the Monte Carlo approximation of the integral in \eqref{norm} induced by these points. This reduction is achieved by taking a gradient descent step, and after each step a new random sample of collocation points is drawn and the process repeated. This approach is an extension of the Deep Galerkin Method first introduced for static parameter values in \cite{sir17}.
\\ \indent
We will take $\pi$ and $\pi^b$ to be uniform on $\Omega$ and $\partial \Omega$ respectively throughout this work. The parameter measure $\pi^\alpha$ however is a key ingredient whose specification can have a significant impact on both the accuracy of the surrogate and the efficiency of its training. It has been shown that the number of neural network parameters required to accurately approximate the parametric solution to several classes of PDE depends only on the intrinsic dimension of the solution manifold \cite{kut19}. This is in contrast to a dependency on the dimensionality of PDE parameters themselves, meaning solutions approximated using this method are capable of overcoming the curse of dimensionality associated with large numbers of parameters if the solution over the training domain has a low intrinsic dimension. The measure $\pi^\alpha$ influences the intrinsic dimensionality of the solution manifold and so its specification is important. 
\\ \indent
Previous work has set $\pi^\alpha$ to be a uniform distribution on a compact subset of $\mathbb{R}^M$ \cite{teo_deepsurrogate}, producing a general surrogate, that is a surrogate trained with $\pi^\alpha$ chosen independently of any data and spanning a relatively large parameter space. This has the advantage that the surrogate need only be trained once, resulting in an analytic function that can be stored and applied to various datasets, however this surrogate will typically be expensive to train and less accurate than a more focused measure. In Section \ref{sec:sim_study}, we demonstrate that this is the case in our setting, and that using a general surrogate leads to inaccurate inferences for this problem. An alternative approach is to set $\pi^\alpha$ to be the posterior distribution. By definition, this measure weights the training across  parameter space exactly as we require, and since this is concentrated around parameters that achieve a good fit to the data the corresponding solution manifold will have a lower intrinsic dimension. Of course this approach assumes that we know the posterior distribution a-priori, which is not the case in practice, and thus requires an adaptive training regime that trains the surrogate and approximates the posterior simultaneously.

\subsubsection*{Data adaptive training}
\label{maptraining}
As with previous adaptive surrogate methods \cite{li2014adaptive}, we seek an accurate approximation over the posterior, thus our approach to training the adaptive surrogate seeks to set the parameter measure $\pi^\alpha$ in the loss function \eqref{loss} to be the posterior distribution \eqref{bayesformula_alph}. This is ultimately achieved by using MCMC samples from the posterior as training points for the surrogate. However, this approach raises the issue of how to initially generate these samples given that we require a trained surrogate to begin MCMC. Our solution is to first train the surrogate over the Laplace approximation to the posterior distribution
\begin{align}
    \pi^{\alpha}_{Laplace} =  MVN\left(\boldsymbol{\alpha}^*,\text{Hess}^{-1}_{\boldsymbol{\alpha}}\left\{-\text{log}\ p(\boldsymbol{\alpha}^*,\sigma_\epsilon^*|\hat{t},\hat{x},\hat{z})\right\}\right),
\end{align}
where $\boldsymbol{\alpha}^*,\sigma_\epsilon^*$ are the MAP estimates 
\begin{align}
    (\boldsymbol{\alpha}^*,\sigma_\epsilon^*) = \text{argmax}\left\{\log p(\boldsymbol{\alpha},\sigma_\epsilon|\hat{t},\hat{x},\hat{z})\right\}. \label{map}
\end{align}
The MAP optimisation is carried out by beginning with some initial estimate $(\boldsymbol{\alpha}_0,\sigma_{\epsilon 0})$ and iteratively updating this estimate using gradient ascent. To approximate the gradients of the objective function \eqref{map} at iteration $n$, the deep surrogate approximation is updated to approximate the solution over a local radius surrounding $\boldsymbol{\alpha}_n$ and automatically differentiated. These gradients are applied to update to $\boldsymbol{\alpha}_{n+1}$, and the process is repeated. The size of the local radius and the step size is reduced as we approach the MAP estimate in a manner similar to trust region optimisation (see Appendix \ref{B} material for more detail). Given the MAP estimate $(\boldsymbol{\alpha}^*,\sigma_\epsilon^*)$ and the corresponding locally trained surrogate $\hat{u}^*(t,x,\boldsymbol{\alpha})$, the covariance of the Laplace approximation to the posterior distribution is readily available by automatic differentiation, thus we next train a surrogate over this approximation by using samples from the Laplace approximation as training points. This initial surrogate is then utilised within an MCMC scheme to begin sampling from the posterior distribution.
\begin{figure}[H]
\centerline{\includegraphics[width=4.in]{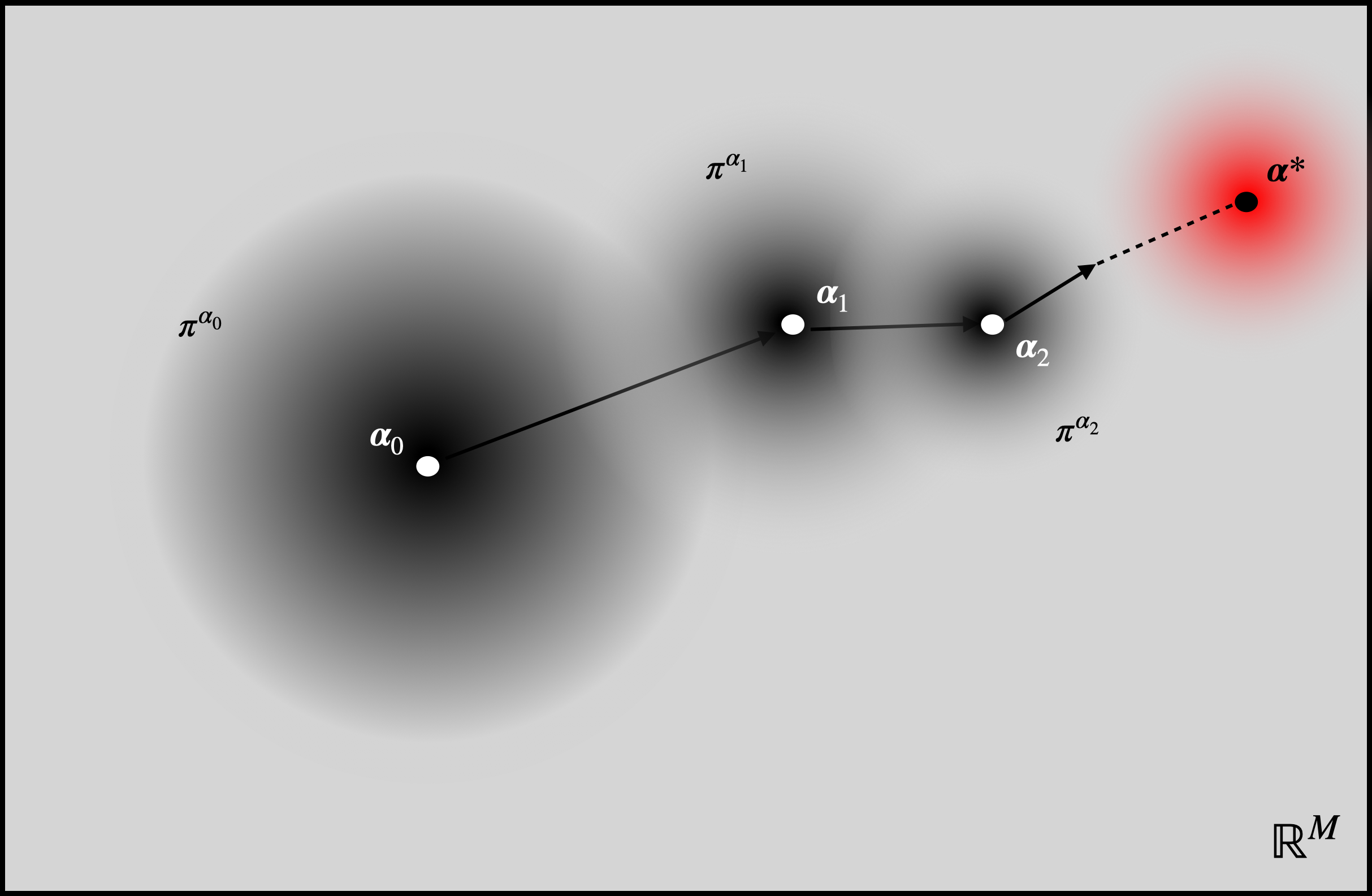}}	
\caption{Illustration of the adaptive training scheme. At iteration $n$ the surrogate is trained locally by minimising the loss \eqref{loss} with local measure $\pi^{\alpha}=\pi^{\alpha_n}$ as depicted by the grey measures. The gradient $\nabla_\alpha \hat{u}$ of the locally trained surrogate is then used to update to $\boldsymbol{\alpha}_{n+1}$ by gradient ascent. Upon convergence to the MAP estimate $\boldsymbol{\alpha^*}$, the surrogate is trained over the Laplace approximation to the posterior shown in red (by minimising \eqref{loss} with $\pi^\alpha = \pi_{Laplace}^\alpha$). The surrogate is then ready to begin sampling. }
\label{fig:adapttrain}
\end{figure}
Although the Laplace approximation provides a reasonable estimate of the posterior distribution for many applications, it may be inaccurate if the true posterior deviates too far from a Gaussian. To account for this and ensure that the training measure $\pi^{\boldsymbol{\alpha}}$ is as close as possible to the posterior distribution, we continue training the surrogate online using posterior samples from the MCMC scheme as training points for the loss function \eqref{loss}. This additional refinement of the approximation is interleaved with the MCMC sampler during a warm-up period, after which the surrogate is fixed and used to sample from the posterior distribution.
\\ \indent
It is possible to understand the relation between the training loss function and the accuracy of the posterior distribution induced by the adaptively trained surrogate in the Hellinger metric. A full proof of this relationship is provided in Appendix \ref{A}. Here we give an approximate argument to illustrate the main features of the proof. Let us denote the true posterior by $\psi(\boldsymbol{\theta})$ and let $\psi_0(\boldsymbol{\theta})$ be the prior distribution so that 
\begin{align}
    d\psi = \frac{1}{Z}\exp(-\Phi(\boldsymbol{\theta}))\,d\psi_0(\boldsymbol{\theta}).
\end{align}
Here  $\boldsymbol{\theta}=(\boldsymbol{\alpha},\sigma_\epsilon)$,  $\Phi(\boldsymbol{\theta})= \frac{1}{2\sigma_\epsilon ^2 }\sum_{n=1}^N( {\hat{z}_n - u(\hat{t}_n,\hat{x}_n,\boldsymbol{\alpha}) })^2$ and $Z$ is a normalisation constant. Denote the approximated posterior induced by substituting the surrogate into the likelihood by $\hat{\psi}(\boldsymbol{\theta})$, and its associated potential by $\hat{\Phi}(\boldsymbol{\theta})$.
\\ \indent
Let $\frac{d\psi}{d\psi_0}=f$ and $\frac{d\hat{\psi}}{d\psi_0}=\hat{f}$ be Radon--Nikodym derivatives with respect to $\psi_0$. Then the squared Hellinger distance between $\psi$ and $\hat{\psi}$ is
\begin{align*}
    H^2(\psi,\hat{\psi}) &= \frac{1}{2}\int (\sqrt{f(\boldsymbol{\theta})}-\sqrt{\hat{f}(\boldsymbol{\theta})})^2d\psi_0(\boldsymbol \theta)\\
    &=\frac{1}{2}\int (\sqrt{f(\boldsymbol{\theta})/\hat{f}(\boldsymbol{\theta})}-1)^2d\hat{\psi}(\boldsymbol\theta). \numberthis
\end{align*}
Assuming $\Phi(\boldsymbol{\theta})\approx \hat{\Phi}(\boldsymbol{\theta})$ and $Z\approx\hat{Z}$, we can approximate
\begin{align*}
    H^2(\psi,\hat{\psi}) &= \frac{1}{2}\int \left(\frac{\sqrt{\hat{Z}}}{\sqrt{Z}}\text{exp}\left(\frac{1}{2}(\Phi(\boldsymbol{\theta})-\hat{\Phi}(\boldsymbol{\theta}))\right)-1\right)^2d\hat{\psi}(\boldsymbol \theta)\\
    &\approx \frac{1}{8}\int \left(\hat{\Phi}(\boldsymbol{\theta})-\Phi(\boldsymbol{\theta})\right)^2d\hat{\psi}(\boldsymbol \theta). \numberthis
\end{align*}
By assuming that $u$ and $\hat u$ are bounded uniformly, we have for some $M>0$ that
\begin{align*}
    |\Phi(\boldsymbol{\theta})-\hat{\Phi}(\boldsymbol{\theta})| &= \left|\frac{1}{2\sigma_\epsilon ^2 }\sum_{n=1}^N( {\hat{u}(\hat{t}_n,\hat{x}_n,\boldsymbol{\alpha}) - u(\hat{t}_n,\hat{x}_n,\boldsymbol{\alpha}) })( \hat{u}(\hat{t}_n,\hat{x}_n,\boldsymbol{\alpha}) + u(\hat{t}_n,\hat{x}_n,\boldsymbol{\alpha})-2\hat{z}_n)\right| \\
    &\leq \frac{M}{2\sigma_\epsilon ^2 }\sum_{n=1}^N|( {\hat{u}(\hat{t}_n,\hat{x}_n,\boldsymbol{\alpha}) - u(\hat{t}_n,\hat{x}_n,\boldsymbol{\alpha}) })|, \numberthis
\end{align*}
using $(a-c)^2 - (b-c)^2 = (a-b)(a+b-2c)$. So applying Jensen's inequality we have that the Hellinger distance approximately satisfies
\begin{align}
    H^2(\psi,\hat{\psi}) \leq \frac{NM^2}{32}\int\frac{1}{\sigma_\epsilon ^4 }\sum_{n=1}^N( {\hat{u}(\hat{t}_n,\hat{x}_n,\boldsymbol{\alpha}) - u(\hat{t}_n,\hat{x}_n,\boldsymbol{\alpha}) })^2\,d\hat{\psi}(\boldsymbol \theta).
\end{align}
Now, $\hat{u}$ satisfies a second-order PDE similar to \eqref{fineq} with perturbed initial and boundary data and an extra forcing term (by substituting $\hat u$ into \eqref{fineq}). We may apply $L^2$-stability theory for second-order parabolic equations (e.g., \cite{Renardy2006-dt}), to determine that,  
\begin{align*}
\frac{1}{N}    \sum_{n=1}^N( {\hat{u}(\hat{t}_n,\hat{x}_n,\boldsymbol{\alpha}) - u(\hat{t}_n,\hat{x}_n,\boldsymbol{\alpha}) })^2
    &\approx \frac{1}{A}
     \|\hat{u}(\cdot,\cdot,\boldsymbol\alpha)-u(\cdot,\cdot,\boldsymbol\alpha)\|^2_{L_2(\Omega)} \\
    &\leq F(\boldsymbol\alpha), \numberthis
\end{align*}
where $F$ is the size of the perturbed boundary data and extra forcing term as defined in \eqref{loss} with appropriate weighting $\nu_1,\nu_2$. Here we have approximated $\frac{1}{N}\sum_{n=1}^N( \hat{u}(\hat{t}_n,\hat{x}_n,\boldsymbol{\alpha}) - u(\hat{t}_n,\hat{x}_n,\boldsymbol{\alpha}) )^2$ by $\frac{1}{A}\|\hat{u}-u\|^2_{L_2(\Omega)}$, where $A$ is the area of $\Omega$, which is reasonable in our settings where data is collected at regular space-time intervals. Then, for some $C>0$,
\begin{align}
H^2(\psi,\hat\psi) \leq C\int\frac{1}{\sigma_\epsilon^4} F(\boldsymbol\alpha)\,d\hat\psi(\boldsymbol{\alpha},\sigma_\epsilon). \label{rside}
\end{align}
The right-hand side corresponds to the training loss function \eqref{loss} when the prior on the hyper-parameter $\sigma_\epsilon$ is a delta distribution (or equivalently if $\sigma_\epsilon$ is treated as known). When evaluated with MCMC samples from the surrogate posterior, it provides an approximate bound on the posterior error in the Hellinger metric. For non-trivial priors on $\sigma_\epsilon$ our Monte Carlo approximation to the loss function \eqref{loss} can be generalised to apply to the right side of \eqref{rside} by weighting the contribution of each term in the approximation by the corresponding MCMC samples for $\sigma_\epsilon^{-4}$.



\subsection{Improving the posterior approximation using delayed-acceptance HMC}
\label{sec:mcmc}
The main aim in a Bayesian inverse problem is to infer the posterior distribution. Most commonly for PDE-based problems these approaches will be based upon optimisation of parameterised approximations to the target distribution. Such methods --- which include variational techniques and Laplace approximations --- have the advantage of being numerically tractable using traditional techniques for PDEs. However this approximation to the target distribution  is based upon assumptions made about its shape a-priori. Alternatively, Markov chain Monte Carlo methods empirically approximate the target by sampling directly from it. The resulting approximations can be made arbitrarily accurate by the law of large numbers, but require a large number of decorrelated samples. This requirement has restricted the application of MCMC to PDE-based inverse problems as each MCMC iteration requires the PDE to be solved numerically, which is an expensive undertaking. 
\\ \indent
The key advantages of deep neural network approximations to the parametric solutions of PDEs are that they can efficiently be evaluated and differentiated with respect to the PDE parameters using automatic differentiation within dedicated software such as TensorFlow \cite{tensorflow}. This ease of manipulation of the deep surrogate model affords us great flexibility in how we approximate the posterior distribution. For example, one can quickly establish a Laplace approximation as described in Section \ref{maptraining}, or use their differentiability to efficiently implement gradient-based MCMC schemes such as Hamiltonian Monte Carlo (HMC) \cite{neal2011mcmc}. Conversely, a current disadvantage of deep learning-based PDE solvers is a lack of convergence guarantees. Despite the remarkable approximation capabilities of neural networks for many classes of PDEs, our ability to realise these approximations typically relies on non-convex optimisation using stochastic gradient descent, and as a result we are currently unable to guarantee the accuracy of deep learning surrogates. 
\\ \indent
In the following we introduce the sampling scheme that we will apply to infer $\hat{Bi}(t,x)$. Our scheme efficiently combines the gradient-based proposal distributions defined by the Hamiltonian Monte Carlo method with the guaranteed accuracy of a finite-difference solver using a delayed-acceptance verification step. We begin by describing Metropolis--Hastings MCMC, and then highlight how a Hamiltonian delayed-acceptance proposal distribution can be incorporated to improve sampling efficiency and accuracy.

\subsubsection*{Metropolis--Hastings Markov chain Monte Carlo}
A Markov chain is a discrete-time stochastic process $X_0,X_1,X_2,\dots$ such that $X_{n+1}$ given $X_n$ is independent of all other previous states. In our application we wish to sample from the posterior distribution over the parameters $\boldsymbol{\theta}=(\boldsymbol{\alpha},\sigma_\epsilon)$, accordingly let us consider a stochastic process  $\boldsymbol{\theta}_0,\boldsymbol{\theta}_1,\boldsymbol{\theta}_2,\dots$ taking values in the parameter space $\Theta \subset \mathbb{R}^M \times\mathbb{R}^{+}$. If the process satisfies the Markov property
\begin{align}
g(\boldsymbol{\theta}_{n+1}|\boldsymbol{\theta}_n, ..., \boldsymbol{\theta}_0) = g(\boldsymbol{\theta}_{n+1}|\boldsymbol{\theta}_n), \label{transition}
\end{align}
then it is a Markov chain, and $g(\boldsymbol{\theta}_{n+1}|\boldsymbol{\theta}_n)$ is the transition kernel of the process (the probability density given $\boldsymbol{\theta}_n$, of transitioning from $\boldsymbol{\theta}_{n}$ to $\boldsymbol{\theta}_{n+1}$). The stationary distribution of a Markov chain is a probability measure $\psi$ such that $\psi(\boldsymbol{\theta}) \geq 0$  $\forall \boldsymbol{\theta} \in \Theta,\ \int_\Theta d\psi(\boldsymbol{\theta}) = 1$ and $\int_Bd\psi(\boldsymbol{\theta}) = \int_\Theta g(B|\boldsymbol{\theta})d\psi(\boldsymbol{\theta})$ for any $B\subset \Theta$. Here the left hand side is the probability of $B$, whereas the right side is the probability that transitions into $B$ from $\psi$ in one time step. Their equality implies that $\psi$ remains constant after transitioning. 
\\ \indent
Markov chain Monte Carlo methods are a group of algorithms for sampling from a distribution known up to an arbitrary constant of proportionality. This is achieved by simulating ergodic Markov chains for which the target distribution is their stationary distribution. The states visited by the process then form a sample from this distribution and the empirical density of these states converges in distribution to the target. A sufficient condition for constructing an appropriate Markov chain is detailed balance, which states that $g(\boldsymbol{\theta'}|\boldsymbol{\theta})\psi(\boldsymbol{\theta}) = g(\boldsymbol{\theta}|\boldsymbol{\theta'})\psi(\boldsymbol{\theta'})$ holds for any $\boldsymbol{\theta},\boldsymbol{\theta'}\in \Theta$. 
It can  be shown that if $(g, \psi)$ satisfies detailed balance, then $\psi$ is the stationary distribution of the process. In MCMC methods we construct transition kernels that achieve detailed balance for a given distribution. In particular if we choose the target distribution as $\psi(\boldsymbol{\theta}) = p(\boldsymbol{\theta}|\hat{t},\hat{x},\hat{z})$  defined in \eqref{bayesformula_alph} then this allows us to sample from the posterior. One of the simplest approaches to the construction of appropriate transition kernels is the Metropolis--Hastings (MH) algorithm. 
\\ \indent
In the Metropolis--Hastings algorithm we choose a proposal density $q({\boldsymbol{\theta}_{prop}}|{\boldsymbol{\theta}})$ describing the probability of proposing a transition from ${\boldsymbol{\theta}}$ to $\boldsymbol{\theta}_{prop}$. This proposal is then accepted with probability
\begin{align}
A_{M}(\boldsymbol{\theta}_{prop},\boldsymbol{\theta}) = \text{min}\left\{1, \frac{q({\boldsymbol{\theta}}|{\boldsymbol{\theta}_{prop}}) \psi(\boldsymbol{\theta}_{prop})}{q({\boldsymbol{\theta}_{prop}}|{\boldsymbol{\theta}}) \psi(\boldsymbol{\theta})}\right\}. \label{met-acc}
\end{align}
It is readily verifiable that detailed balance is preserved with respect to $\psi$ if the transitions are carried out this way. The full algorithm can be written as: \\
\begin{algorithm}[H]
\caption{Metropolis-Hastings}
\label{alg:buildtree}
\begin{algorithmic}
\STATE{Choose initial $\boldsymbol{\theta}_0 \in \Theta$;}
\FOR{$i = 0,1,\dots,N_{samples}$}
\STATE{Propose $\boldsymbol{\theta}_{prop} \sim q({\boldsymbol{\theta}_{prop}}|{\boldsymbol{\theta}_i})$;}
\STATE{Sample $u \sim U(0,1)$}
\IF{$u \leq A_{M}(\boldsymbol{\theta}_{prop},\boldsymbol{\theta}_i)$;}
\STATE{set $\boldsymbol{\theta}_{i+1} := \boldsymbol{\theta}_{prop}$ (accept);}
\ELSE
\STATE{set $\boldsymbol{\theta}_{i+1} := \boldsymbol{\theta}_i$ (reject);}
\ENDIF
\ENDFOR
\RETURN $\boldsymbol{\theta} = (\boldsymbol{\theta}_0,\boldsymbol{\theta}_1,\dots,\boldsymbol{\theta}_{N_{samples}})$
\end{algorithmic}
\end{algorithm}
\noindent 
The choice of proposal distribution $q(\boldsymbol{\theta}_{prop}|\boldsymbol{\theta})$ is key to the success of MH. An efficient proposal distribution will yield samples that are not heavily correlated, resulting in a sample that estimates the posterior with a lower statistical error. For a sample generated using some proposal distribution, the effective sample size (ESS) quantifies the efficiency of the proposal strategy by returning the number of i.i.d samples required to reach the same level of precision \cite{gey92}.  In practice simple symmetric proposal distributions such as the multivariate Gaussian distribution centred at $\boldsymbol{\theta}$ are commonly used. This random walk Metropolis--Hastings (RWMH) scheme has the advantage of simplicity, as it requires just one PDE solve per iteration and symmetric distributions cancel out in the acceptance ratio \eqref{met-acc}, making them simpler to implement. Unfortunately exploration strategies based on symmetric random walks such as this typically result in Markov chains that have a low acceptance rate and highly correlated samples in higher dimensions, leading to a low ESS and slow convergence to the stationary distribution. 
\\ \indent
Alternative approaches --- such as Hamiltonian Monte Carlo (HMC) and Metropolis Adjusted Langevin (MALA) --- use gradients of the target distribution to improve the efficiency of MCMC. Put briefly, in HMC we define a Hamiltonian function using the target distribution by
\begin{align*}
    H(\boldsymbol{\theta},\boldsymbol{p})  &= - \log \psi(\boldsymbol{\theta}) -\log \psi(\boldsymbol{p}|\boldsymbol{\theta}) \\
&=U(\boldsymbol{\theta}) + K(\boldsymbol{p}|\boldsymbol{\theta}). \numberthis
\end{align*}
Here $U(\boldsymbol{\theta}) = - \log \psi(\boldsymbol{\theta})$ represents the potential energy and $K(\boldsymbol{p}|\boldsymbol{\theta}) = -\log \psi(\boldsymbol{p}|\boldsymbol{\theta})$ represents the kinetic energy of the system, where $\psi(\boldsymbol{p}|\boldsymbol{\theta})$ is a user specified distribution over some auxiliary momentum variables $\boldsymbol{p}$ (we take the usual choice of $\psi(\boldsymbol{p}|\boldsymbol{\theta})\sim MVN(\boldsymbol{0},M)$ with some user defined `mass matrix' $M$ \cite{neal2011mcmc}) . Proposals are then generated by using symplectic integrators to simulate random trajectories from the corresponding Hamiltonian system. This has the effect of significantly reducing the autocorrelation of the chain. It is impractical to perform HMC using traditional PDE solvers since $\nabla_\theta\psi(\boldsymbol{\theta})$ is required at each step of the symplectic integrator, leading to tens or hundreds of PDE solves being performed for each MCMC sample. MALA, by comparison, simulates the Langevin dynamics of a process that has stationary distribution $\psi(\boldsymbol{\theta})$, and can be shown to be equivalent to performing HMC with trajectories simulated over just one time step. MALA therefore represents a more tractable intermediary between random walk and Hamiltonian proposals, requiring just two gradients per iteration which can be calculated using adjoint PDE methods. This yields some improvement over symmetric random walk proposals, however we demonstrate that it is significantly outperformed by HMC, thus justifying the need for efficient HMC methods for PDEs.  

\subsubsection*{Delayed-acceptance HMC}
With deep surrogate models the burden of repeated differentiation of the PDE solution is alleviated by automatic differentiation, therefore HMC can be performed efficiently. Unfortunately the accuracy of deep surrogate solutions cannot be mathematically guaranteed as is the case with traditional numerical solvers such as finite-differences. This inaccuracy has the potential to bias the posterior sample, leading to unreliable estimates that cannot safely be used in turbine design. The delayed-acceptance HMC sampler that we describe here overcomes this by performing Metropolis--Hastings using a finite-difference solver with a proposal distribution constructed such that:
\begin{enumerate}
    \item HMC proposals based on the deep learning surrogate are utilised for efficient exploration of the posterior distribution.
    \item The finite-difference solver is only executed if the proposal is accepted according to the regular HMC acceptance criterion based on the surrogate.
\end{enumerate}
The delayed-acceptance method was first described in \cite{fox_da}. To describe our adaptation of this method, let us first denote the posterior distribution induced by substituting the deep surrogate model $\hat{u}(\hat{t}_n,\hat{x}_n,\boldsymbol{\alpha})$ into the likelihood \eqref{likelihood_alph} by $\hat{\psi}(\boldsymbol{\theta})$, and similarly denote by $\tilde{\psi}(\boldsymbol{\theta})$ the posterior induced by the finite-difference solution $\tilde{u}(\hat{t}_n,\hat{x}_n,\boldsymbol{\alpha})$. 
\\ \indent
In the delayed-acceptance HMC scheme we sample from $\tilde{\psi}(\boldsymbol{\theta})$ using Metropolis--Hastings with proposal distribution given by
\begin{align*}
    q(\boldsymbol{\theta}_{prop}|\boldsymbol{\theta}) = A_{H}(\boldsymbol{\theta}_{prop},&\boldsymbol{\theta})q_{H}(\boldsymbol{\theta}_{prop}|\boldsymbol{\theta}) \numberthis \label{da_kernel} \\ &+ \left(1-\int \left(A_{H}(\boldsymbol{\theta}_{prop},\boldsymbol{\theta})q_{H}(\boldsymbol{\theta}_{prop}|\boldsymbol{\theta}) \right)d\boldsymbol{\theta}_{prop}\right) \delta(\boldsymbol{\theta}-{\boldsymbol{\theta}_{prop}}).
\end{align*}
Here $q_H(\boldsymbol{\theta}_{prop},\boldsymbol{\theta})$ and $A_H(\boldsymbol{\theta}_{prop},\boldsymbol{\theta})$ are the HMC proposal distribution and acceptance probability applied using $\hat{\psi}(\boldsymbol{\theta})$. The first term in \eqref{da_kernel} is therefore the density of HMC proposals that are accepted, while the second term assigns the remaining probability to the chain remaining in place at $\boldsymbol{\theta}$. 
\\ \indent
It is natural to execute this proposal in two stages. During the first, a regular HMC proposal is made with standard acceptance criterion \eqref{met-acc} (applied using the surrogate posterior $\hat{\psi}$), this acts as preliminary screening by the surrogate. If rejected by this criterion, the Markov chain remains in the same state and the finite-difference solver is not executed. If accepted, the proposal is passed to a secondary acceptance criterion with acceptance probability
\begin{align}
        A(\boldsymbol{\theta}_{prop},\boldsymbol{\theta})&=\text{min}\left\{1, \frac{\tilde{\psi}(\boldsymbol{\theta}_{prop})\hat{\psi}(\boldsymbol{\theta})}{\tilde{\psi}(\boldsymbol{\theta})\hat{\psi}(\boldsymbol{\theta}_{prop})}\right\}. \label{delayed_accept_prob}
\end{align}
This additional acceptance criterion ensures detailed balance is achieved with respect to $\tilde{\psi}$, thus the posterior is accurate to the level of the finite-difference solver. One can verify this by substituting \eqref{da_kernel} into the standard Metropolis acceptance criterion \eqref{met-acc} and simplifying (see \cite{ef06} for the full calculation). Additionally, the error between $\tilde{\psi}$ and the true posterior $\psi$ in the Hellinger distance can be bounded by combining standard error analysis for finite differences with upper bounds on posterior errors given approximations to the likelihood potential (as described in Appendix \ref{A}). Therefore samples obtained using this scheme are provably close to the true posterior. 
\\ \indent
The sampling efficiency and preliminary screening offered in the delayed-acceptance HMC algorithm ensures we are able to maximise the utility of the finite-difference solves. This allows us to tractably approximate the posterior by MCMC with the guaranteed accuracy that is critical in practice. There is of course some additional computational effort in comparison to relying purely on the surrogate, which we quantify in our simulation study below, as well as comparing to the accuracy of posterior approximations based purely on surrogates.

\section{Numerical results}
\label{sec:results}
The remainder of this work is dedicated to demonstrating deep surrogate accelerated delayed-acceptance HMC on real and simulated data. We first outline the specifics of our implementation, detailing the choice of prior, network architecture, training strategy and HMC calibration. Using simulated data we then demonstrate the speed and accuracy of our approach in comparison to alternative methods, before concluding this section with some results obtained using data generated by an experimental rig. All of the results and timings presented were produced in TensorFlow using a machine with a mobile RTX 2080 and a 6 core 3.9Ghz processor. The code for this paper is available at \texttt{\url{https://github.com/teojd/surrogate_da-hmc}}.

\subsection{Implementation details} \label{implementation}
\subsubsection*{Prior distribution}
We assign a Gaussian process prior distribution to $Bi(t,x)$ with mean equal to zero $\mu(t,x)=0$, and a separable covariance kernel in space and time
\begin{align}
    K([t,x],[t',x']) = \sigma^2K_x(x,x')K_t(t,t'). \label{seperable}
\end{align}
For the spatial covariance $K_x(x,x')$ we adopt the twice differentiable Mat\'ern kernel
\begin{align}
    K_x(x,x') = \left(1+\frac{|x-x'|}{\rho_x} + \frac{|x-x'|^2}{3\rho_x^2}\right)\text{exp}\left(-\frac{|x-x'|}{\rho_x}\right).
\end{align}
Here we choose a spatial length scale of $\rho_x= b-a = 0.7$, as this choice has been successfully applied to 1-dimensional steady state experiments \cite{biot}. For the temporal covariance we use the squared exponential kernel
\begin{align}
    K_t(t,t') = \text{exp}\left(-\frac{|t-t'|^2}{2\rho_t^2}\right).
\end{align}
Here we choose $\rho_t=900$, this is a reasonable evolution time scale for such problems, representing a temporal correlation length scale of 15 minutes \cite{teohui}. The marginal standard deviation of the prior distribution is $\sigma=100$ which is large enough to encompass the true parameters used in our experiments. Using (\ref{meanss}, \ref{covss}) the corresponding prior distribution over the coefficients is $\boldsymbol{\alpha} \sim MVN(0,\Sigma)$ where $\Sigma$ is chosen such that
\begin{align}
    \sum_{i,j=1}^MT_{i}(t,x)\Sigma_{i,j}T_{j}(t',x') \approx K([t,x],[t',x']).
\end{align}
We resolve this expression using interpolation on the grid of Chebyshev nodes. The accuracy of this approximation is visualised for each of the spatial and temporal kernels in Figure \ref{fig:kernel}. 
\begin{figure}[H]
\centerline{\includegraphics[width=3.8in]{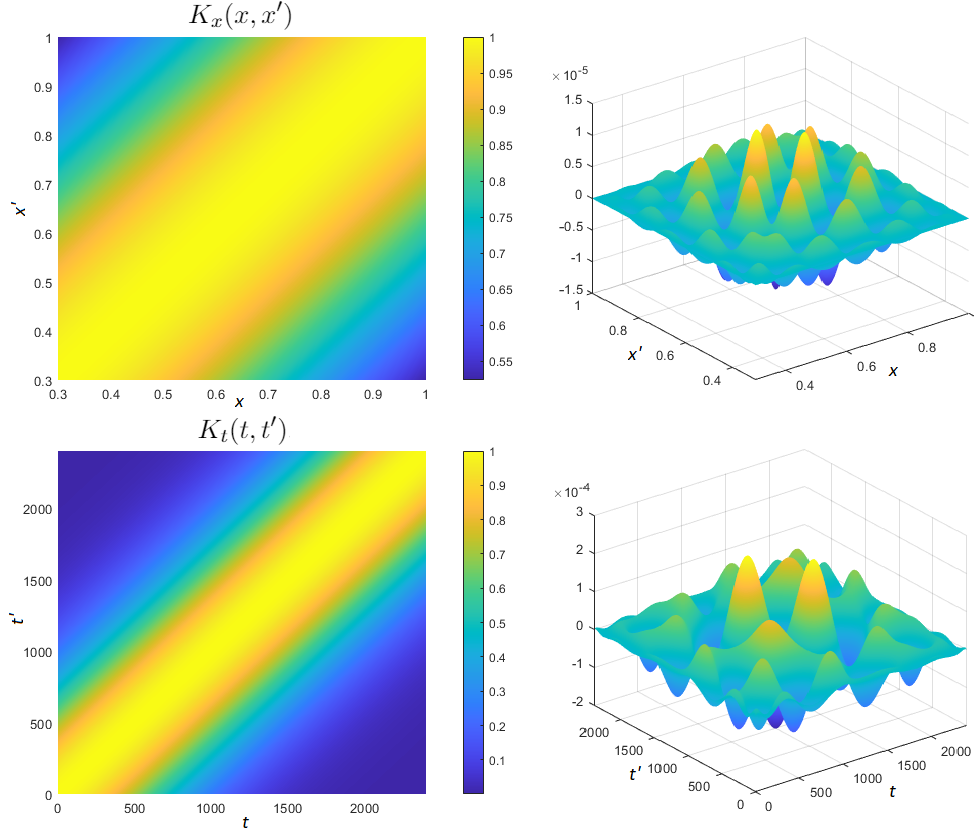}}	
\caption{Left: Chebyshev approximation to $K_x(x,x')$ and $K_t(t,t')$. Right: Error between the true and approximated kernel.}
\label{fig:kernel}
\end{figure}

\subsubsection*{Network architecture}
We use a fully-connected feed-forward neural network consisting of 4 hidden layers of 256 neurons with tanh activation functions to approximate the parametric solution. This deep learning surrogate $\hat{u}:\mathbb{R}^{68}\to \mathbb{R}$ takes as inputs the coordinates $(t,x)\in \mathbb{R}^2$, and the Chebyshev polynomial coefficients $\boldsymbol{\alpha}\in \mathbb{R}^{66}$, and returns an approximation to the corresponding PDE solution at those coordinates.
\subsubsection*{Training}
To train the deep learning surrogate we minimise the loss function \eqref{loss} until no significant decay in the loss function is observed. We train using the ADAM algorithm in TensorFlow \cite{tensorflow} with default settings and a custom learning rate schedule that decays from $3 \times 10^{-3}$ to $10^{-5}$. We assign uniform measures $\pi \sim \text{Unif}(\Omega)$ to the interior domain, and $\pi^b \sim \text{Unif}(\partial \Omega)$ to the boundary domain. For the parameter space measure we compare two approaches. First is an adaptive surrogate computed using the data adaptive training scheme described in Section \ref{sec:surrogate_approx}, which we call the `adaptive surrogate' henceforth due to its adaptation to the posterior induced by the observed data. This is compared to a surrogate trained over $\pi^\alpha = \text{Unif}(A)$ where $A$ is the fixed parameter space $A = \underset{{i+j\leq10}}{\prod} A_{i,j}$, where $A_{i,j} = [-80/2^k,80/2^k]$ for $k=\text{max}(i+j-3,0)$. We refer to the surrogate trained over $A$ as the `general surrogate', as this set is independent of any data and broad enough to cover all of the parameter values that we reasonably expect to infer, whilst also being prevented from becoming unnecessarily complex by the decay in the coefficients of higher degree terms. 


\subsubsection*{HMC calibration}
The surrogate is used to draw 20,000 HMC samples from the posterior distribution, each using 40 leap-frog iterations. The first 10,000 of these constitute a warm-up period during which hyperparameters of the MCMC scheme are automatically determined. In this period the step size of a leap-frog integrator of the Hamiltonian dynamics is calibrated such that an acceptance rate of about 65\% is achieved (on the primary acceptance criteria \eqref{met-acc}). Furthermore the HMC mass matrix is updated such that it is proportional to the inverse of the empirical covariance of the existing sample during this period. Both of these choices are generally regarded as optimal for HMC. After the warm-up period these hyperparameters are fixed, and delayed-acceptance HMC is used with a semi-implicit finite-difference solver to draw a further 10,000 samples from the posterior distribution. As comparisons we also perform RWMH and MALA. Similarly to our HMC implementation we use the warm up period to calibrate their proposal distribution/preconditioning matrix, and adapt the step size to target the corresponding optimal acceptance rates of 23.4\% and 57.4\% respectively. This adaptation is achieved by reacting to the difference between the target acceptance rate and average acceptance rate of the chain during warm up (see equation (20) in \cite{andrieu2008tutorial}). 
\subsection{Simulation study}\label{sec:sim_study}
To validate our method we first provide a study using simulated data. For this we consider the PDE \eqref{fineq} on the domain $(t,x)\in \Omega = [0,3600]\times [0.3,1]$ with $c_0 = 35,000, c_1 = 1, c_2 = 1$, and initial and boundary conditions
\begin{align*}
    u(0,x)   &= x, \ \ \qquad x\in [0.3,1], \\
    u(t,0.3) &= 0.3, \qquad x\in [0.3,1], \\
    u(t,1)   &=  1, \ \ \ \qquad t\in [0,3600]. \numberthis
\end{align*}
We produce noisy data from this PDE by solving the equation for instances of $Bi(t,x)$ using a high resolution FD solver and adding Gaussian white noise to data subsampled from the solution at 152 locations as illustrated in Figure \ref{fig:pde-points}. 
\begin{figure}[H]
\centerline{\includegraphics[width=4.5in]{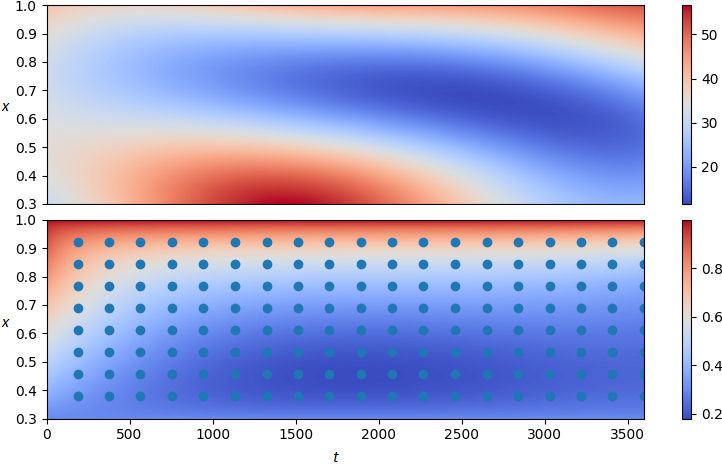}}	
\caption{Top: A random instance of $Bi(t,x)$. Bottom: Corresponding PDE solution and data locations.}
\label{fig:pde-points}
\end{figure}
\noindent
Our aim in the Bayesian inverse problem is to sample from the posterior distribution of $\boldsymbol{\alpha}$ conditional on this data. The instances of $Bi(t,x)$ that we illustrate are direct samples from a grid-based Gaussian process discretisation, that is the true $Bi(t,x)$ are not polynomials in these examples. The ill-posedness of this problem implies that there exist a wide range of Biot numbers that achieve a good fit to the data (see Figure \ref{fig:various-fits} for some examples of this variability for the data in Figure \ref{fig:pde-points}). It is therefore unreasonable to think that we should be able to infer $Bi(t,x)$ with a very high degree of certainty. Instead what we seek is a quantification of this uncertainty that uses our prior distribution to restrict our inference to physically reasonable results.
\begin{figure}[H]
\centerline{\includegraphics[width=4.5in]{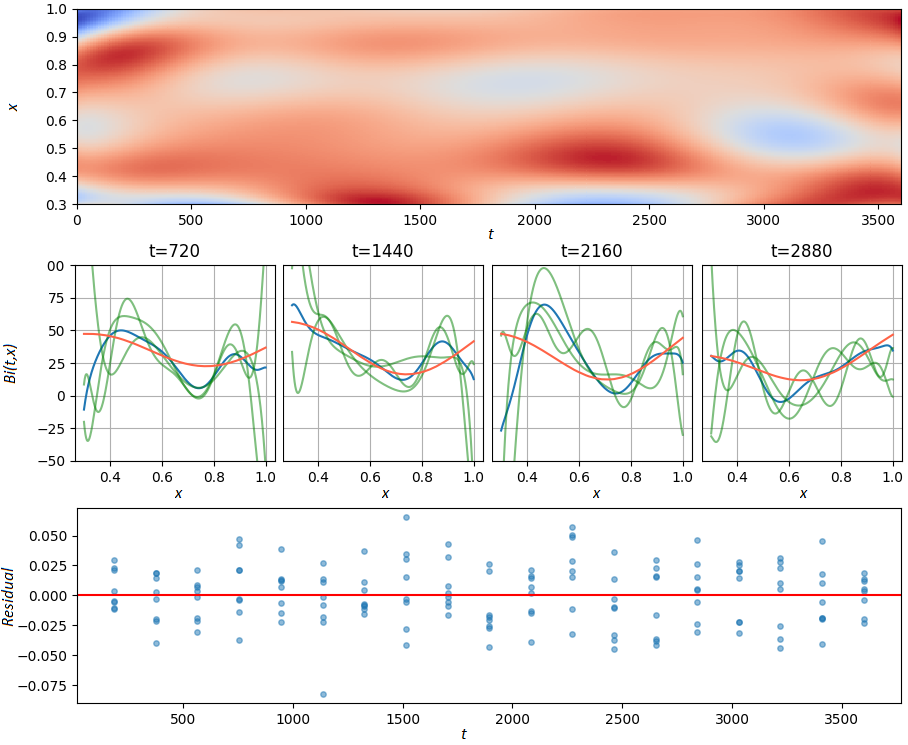}}	
\caption{Top: An example of a $Bi(t,x)$ that is physically unrealistic but achieves a good fit to the data shown in Figure \ref{fig:pde-points}. Middle: At discrete times the spatial profiles of this $Bi(t,x)$ are shown in blue. The profiles of the true $Bi(t,x)$ are plotted in red. The green curves are alternate examples of $Bi(t,x)$ that achieve a similar fit. Bottom: Temporal residuals between the data and solution fitted using this $Bi(t,x)$, demonstrating that the unrealistic $Bi(t,x)$ produces a good fit.}
\label{fig:various-fits}
\end{figure}
\noindent
To train the adaptive surrogate we begin by finding the Laplace estimate. As shown in Figure \ref{fig:laplace}, this provides a reasonable starting point for inferring the Biot number, however at discrete times the spatial profile reveals that the true Biot number used to generate the data is often significantly outside the 95\% credible interval. This behaviour is observed consistently with different instances of $Bi(t,x)$ and so we conclude that the Laplace approximation is overconfident for this problem. This poor quantification of the uncertainty is not unexpected, as the Laplace approximation is a local estimate based only on the posterior curvature at the MAP estimate. In particular, Figure \ref{fig:laplace} justifies the requirement for MCMC methods to quantify the uncertainty more reliably. 
\begin{figure}[H]
\centerline{\includegraphics[width=4.5in]{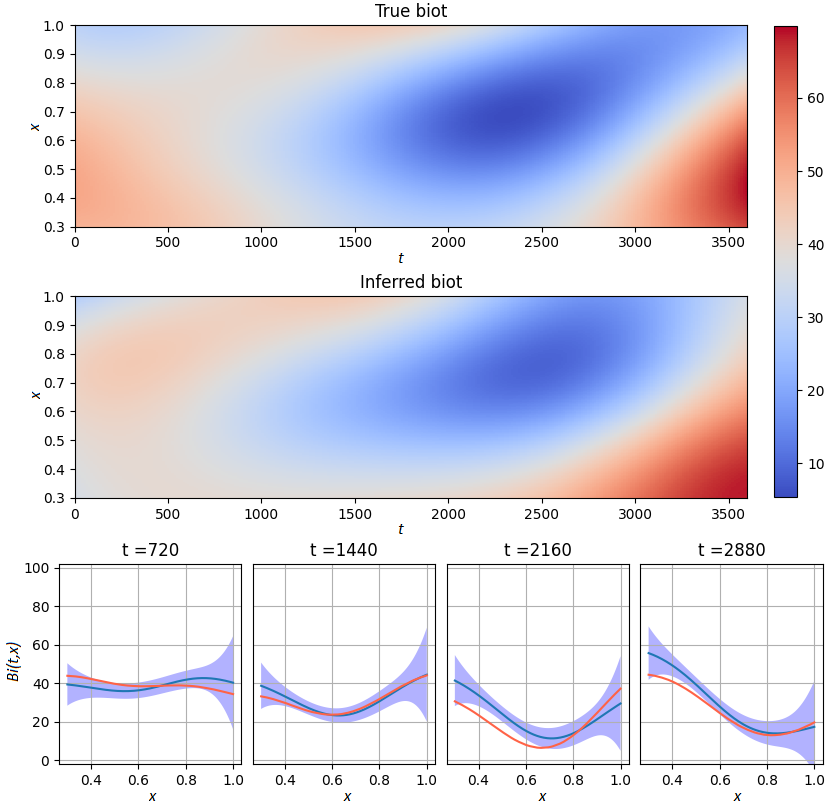}}	
\caption{Laplace estimate of Biot number compared to the true value. In the bottom panel the red curves are spatial profiles of the true $Bi(t,x)$ at discrete times, the blue curves are corresponding profiles of the inferred $Bi(t,x)$. The blue shaded regions are 95\% credible intervals of the Laplace approximation to the posterior.}
\label{fig:laplace}
\end{figure}
\noindent
In our experiments we carry out RWMH, MALA, and HMC using the general and adaptive surrogates. After a warm up period each of these schemes are run for 10,000 iterations, both in a surrogate-only setting, and with delayed-acceptance using finite-differences. In each case, Table \ref{table} shows the accuracy of the surrogate, time taken in seconds for training and sampling, as well as the effective sample size and cost (measured as the time per effective sample). In this table the timings, ESS, and costs given are the averages over runs with 5 different datasets, and the $L^1$ errors correspond to the averaged difference between the surrogate and a high resolution finite-difference solver at the MAP parameter values. In addition, minimum and maximum ESS values are also reported to give an idea of the variability in these 5 runs. Another experiment we carried out was to isolate Monte Carlo variability (without confounding data variability) for the adaptive surrogate. This was tested by fixing a single dataset, then running each scheme 5 times on this dataset. This experiment gives similar outputs to Table \ref{table}, so for brevity these results are omitted, though it is interesting to note that the Monte Carlo variability is consistent in our experiments with or without data variability.

\begin{table}[H]
\centerline{\begin{tabular}{ll|l|l|l|l|l|l|} 
\cline{3-8}& & \multicolumn{3}{l|}{
\textbf{Surrogate-only}} & \multicolumn{3}{l|}{\textbf{Delayed-acceptance}} \\ \hline
\multicolumn{1}{|l|}{
\textbf{Surrogate type}} & \textbf{Proposal} & \textbf{Time} & \textbf{ESS} & \textbf{Cost} & \textbf{Time} & \textbf{ESS} & \textbf{Cost}  \\ \hline
\multicolumn{1}{|l|}{
\textbf{Adaptive}} & RWMH & 71 & 33\ \ \ \ \ \ \ \ \ \ (25,37) & 2.151 & 673 & 20\ \ \ \ \ \ \  (16,26) & 33.216 \\ \cline{2-8} 
\multicolumn{1}{|l|}{
\textit{$t=925s$}} & MALA & 111 & 427\ \ \ \ \ \ \  (331,617) & 0.259 & 800 & 365\ \ \ \  (264,476) & 2.191 \\ \cline{2-8} 
\multicolumn{1}{|l|}{
\textit{$e=5.02\times 10^{-4}$}} & HMC & 380 & 11263\ (10398,12484) & 0.034 & 1338 & 6497\ (5249,7334) & 0.206\\ \hline
\multicolumn{1}{|l|}{
\textbf{General}} & RWMH & 69 & 37\ \ \ \ \ \ \ \ \ \ (24,48) & 1.873 & 665 & 24\ \ \ \ \ \ \ (17,32) & 27.978 \\ \cline{2-8} 
\multicolumn{1}{|l|}{
\textit{$t=15,874s$}} & MALA  & 109 & 372\ \ \ \ \ \ \ (285,420) & 0.294 & 797 & 27\ \ \ \ \ \ \ (18,46) & 29.694 \\ \cline{2-8}
\multicolumn{1}{|l|}{
\textit{$e=1.14\times 10^{-2}$}} & HMC & 389 & 10832 (9725,11760) & 0.036 & 1360 &18 \ \ \ \ \ \ \ (2,40) & 74.499 \\ \hline
\end{tabular}}
\caption{Average timings, accuracy, and ESS achieved by each surrogate type after running 10,000 iterations of each sampling scheme. Here $t$ denotes the training time, and $e$ denotes the $L_1$ error of the surrogate. To give an indication of variability of each sampler, in brackets we show the minimum and maximum ESS obtained from our 5 runs.}
\label{table}
\end{table}
Beginning from a randomly initialised network, the time taken to train the general surrogate is 4 hours and 25 minutes, which is significantly longer than the 15 minutes taken to train the adaptive surrogate. In principle the general surrogate has the advantage of being applicable to multiple datasets once trained, however the larger error of this surrogate prevents it from being a reliable option. This is because gradient-based proposals generated using the general surrogate are almost always rejected by the secondary acceptance criterion \eqref{delayed_accept_prob} in delayed-acceptance, as these are biased towards high-probability regions according to the surrogate-induced posterior, which are typically relatively lower-probability regions according to the FD-induced posterior due to the discrepancy between the general surrogate and the FD solution. Figure \ref{fig:histograms} shows that the inferences obtained by the general surrogate without delayed-acceptance are also inaccurate. If available however, a general surrogate can significantly accelerate the training of an adaptive surrogate if used as the network's initialisation. This acceleration can even apply if the general surrogate was trained to solve a version of the PDE with different boundary conditions and coefficients $\{c_0,c_1,c_2\}$; for example it takes 171 seconds on average to train an adaptive surrogate to the same level of accuracy when initialised by a general surrogate trained to solve the PDE defined for real data in Section \ref{sec:exp_data} with boundary conditions shown in Figure \ref{fig:bc}. 
\\ \indent
Comparing sampling schemes we see that RWMH has the highest cost, followed by MALA. HMC significantly outperforms these methods despite taking longer to perform each iteration, obtaining an ESS that is slightly larger than the number of samples in the surrogate-only setting. The introduction of the additional acceptance criterion in delayed-acceptance lowers the ESS since it reduces the number of transitions that take place. Moreover, the introduction of delayed-acceptance decreases the average speed of each iteration due to the FD solver being executed. It is therefore notable that in the adaptive case, the cost of DA-HMC is still lower than surrogate-only RWMH or MALA, highlighting that the optimisation of the proposal strategy is a significant factor that, if made efficient, can outweigh numerical solver inefficiencies when performing PDE-based MCMC sampling. As a comparison, if we were to apply these sampling schemes solely using a FD solver the estimated cost is 30.40 for RWMH, 4.68 for MALA, and 8.97 for HMC. These estimates are based on the cumulative cost attributable to solving the forward and adjoint PDEs required to calculate the solutions and gradients required for each MCMC scheme, assuming that the ESS achieved is the same as the corresponding adaptive surrogate-only sampler. Although these are somewhat conservative estimates that discard any computational effort in MCMC that is not a direct result of solving the PDEs, this FD-only approach still has a cost that is over 20 times higher than our delayed-acceptance HMC scheme (which converges to the same posterior), and 130 times higher than our surrogate-only approach (which converges to a surrogate-induced approximation).
\\ \indent
To investigate the accuracy of the posteriors achieved using each method we use delayed-acceptance HMC sampling as a baseline, as this is provably accurate and achieves a sufficiently large ESS. Figure \ref{fig:cred-ints} compares the credible intervals achieved by each proposal distribution in the surrogate-only setting. Figure \ref{fig:da-vs-not} compares the credible intervals achieved with and without delayed-acceptance using adaptive and general surrogates. Example trace plots for the coefficient $\alpha_3$, to the term that is cubic in space and constant in time are shown in Figure \ref{fig:traces}, these visualise the samples achieved by each delayed-acceptance scheme over a fixed runtime. Finally in Figure \ref{fig:histograms} for the same coefficient, we show how these differences in proposal distribution and surrogate choice affect the posterior accuracy by comparing the densities of their samples. These results show that in our experiments the adaptive surrogate HMC approach without delayed-acceptance achieves results that are very close to the finite-difference induced posterior obtained using delayed-acceptance (which is in turn close to the true posterior), albeit this behaviour is not mathematically guaranteed a-priori. All other approaches perform poorly in comparison. For RWMH and MALA with an adaptive surrogate this is due to their lower ESS. For the general surrogate the lower accuracy results in a sampling bias in the surrogate-only setting, resulting in $Bi(t,x)$ being underestimated in comparison to the true posterior. Moreover for the general surrogate, when coupled with delayed-acceptance it exhibits an extremely high rejection rate, resulting in chains that do not approach stationarity within the 10,000 iteration run, resulting in extremely poor approximation of the posterior. These examples were included for completeness of experiments, however they are not recommended. 
\begin{figure}[H]
\centerline{\includegraphics[width=5.in]{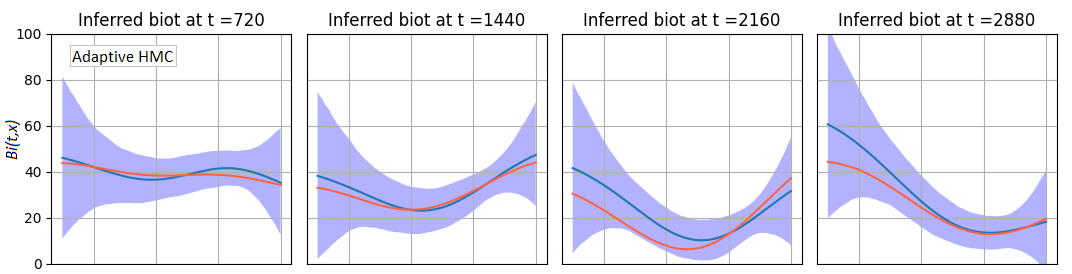}}	
\centerline{\includegraphics[width=5.in]{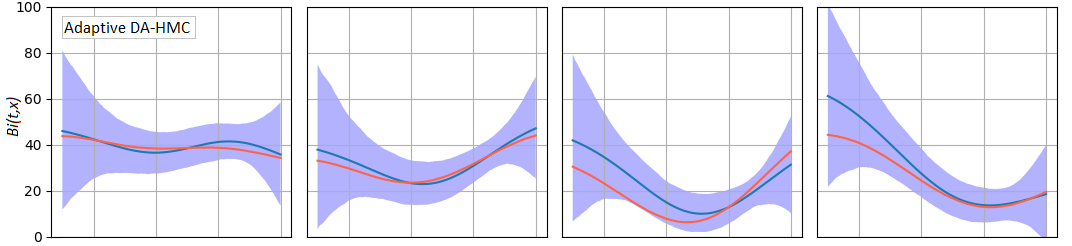}}	
\centerline{\includegraphics[width=5.in]{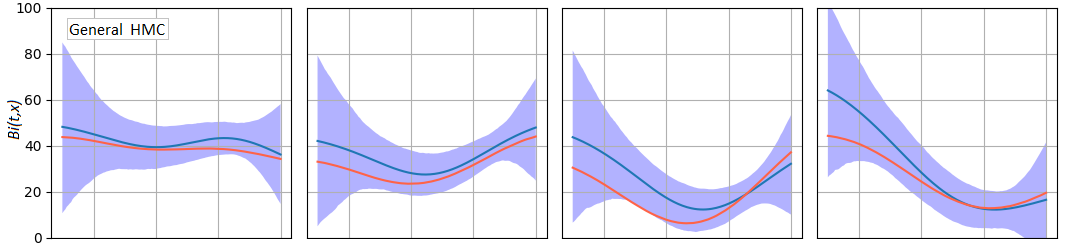}}	
\centerline{\includegraphics[width=5.in]{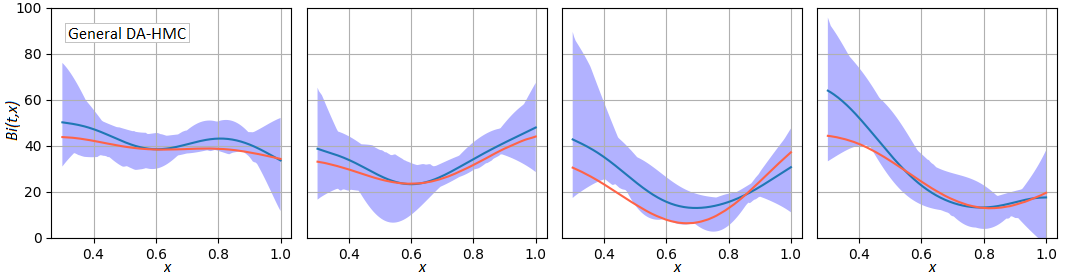}}	
\caption{Spatial profiles of the 95\% credible intervals of surrogate-only adaptive HMC, delayed-acceptance adaptive HMC, surrogate-only general HMC, delayed-acceptance general HMC. The red curves are the spatial profiles of the true Biot number used to simulate the data. Respectively, the blue curves and shaded regions show the corresponding means and 95\% credible intervals given by the approximated posterior achieved by each sampling scheme. } 
\label{fig:da-vs-not}
\end{figure}
\begin{figure}[H]
\centerline{\includegraphics[width=5.in]{adapt_hmc_cross.png}}	
\centerline{\includegraphics[width=5.in]{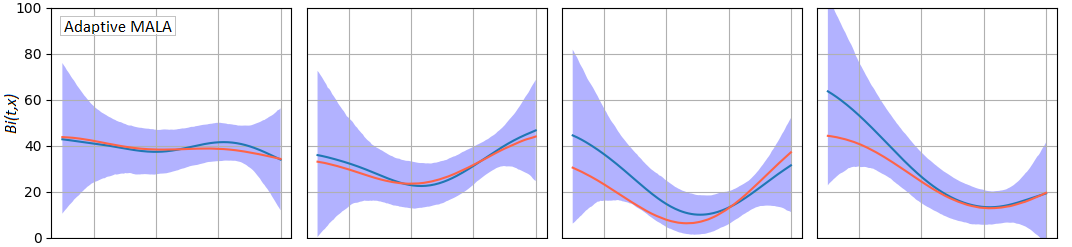}}	
\centerline{\includegraphics[width=5.in]{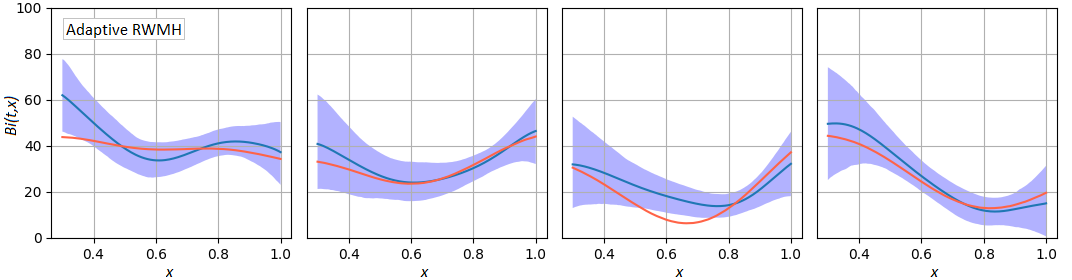}}	
\caption{Spatial profiles of the 95\% credible intervals of surrogate-only adaptive HMC, MALA, and RWMH at various time points. The red curves are the spatial profiles of the true Biot number used to simulate the data. Respectively, the blue curves and shaded regions show the corresponding means and 95\% credible intervals given by the approximated posterior achieved by each sampling scheme.}
\label{fig:cred-ints}
\end{figure}

\begin{figure}[H]
\centerline{\includegraphics[width=5.in]{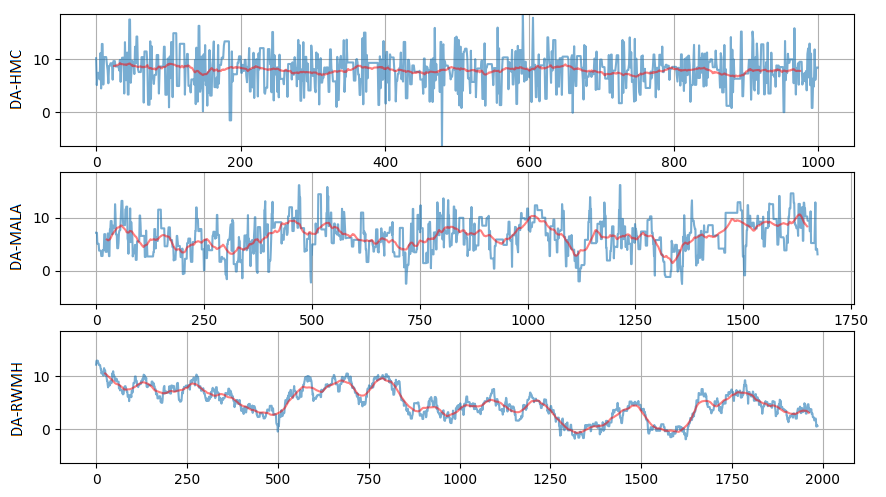}}	
\caption{Comparison of trace plots of adaptive delayed-acceptance HMC, MALA, and RWMH and their 50 step moving average. The lengths of these chains is chosen so that the time taken to achieve each sample is the same.}
\label{fig:traces}
\end{figure}

\begin{figure}[H]
\centerline{\includegraphics[width=4.5in]{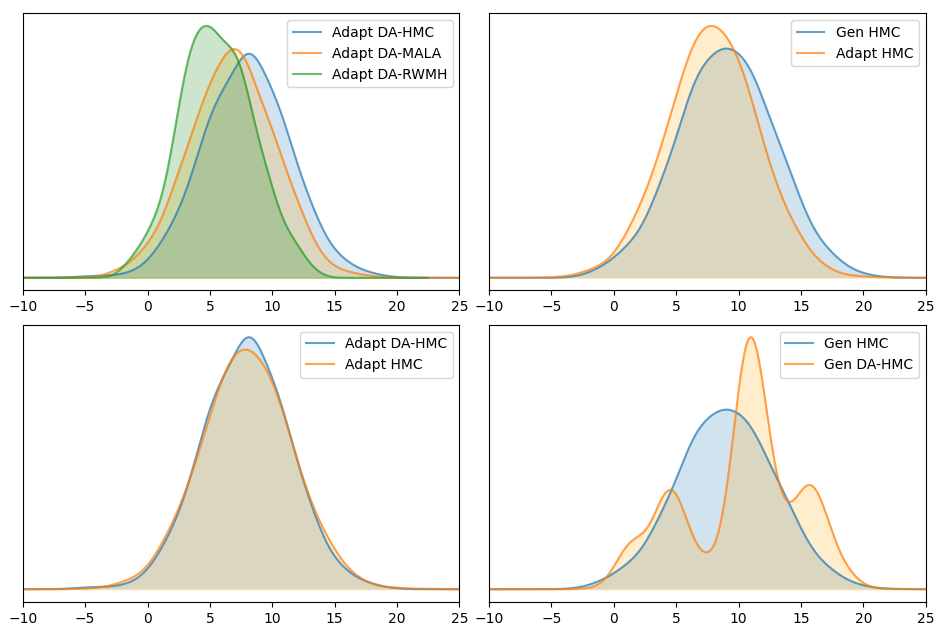}}	
\caption{Comparison of sample densities of different schemes for the Chebyshev coefficient that is cubic in space and constant in time.}
\label{fig:histograms}
\end{figure}
\subsection{Experimental data study}\label{sec:exp_data}
A current aim of engineers in the area of turbo-machinery is to develop well-informed models of the temporal heat transfer within engine cavities. This will allow time-dependent heat transfer scenarios such as take-off and landing of aircrafts to be better understood, and the resulting models will inform design optimisations that will improve engine efficiency and reduce emissions. A key method used to develop the required understanding is through the analysis of experimental data from a range of scenarios. In this section, we apply our methodology to infer the Biot number evolution using real data from a simple experimental setting of a compressor disc rotating in a compressor cavity rig. 
\\ \indent
Detailed information about the experimental setting is given in \cite{teohui}. To summarise, in this experiment a disc spins at 6000RPM in a closed cavity. At the outer radius $b = 0.228$ metres, the disc is heated to 91\textdegree{}C and maintained at that temperature for 500 seconds. At the inner radius $a = 0.133$ cool air flows through the cavity. After 500 seconds the heater is turned off and the disc continues to rotate while its temperature is monitored in 30 second increments at 18 radial locations for a further 1900 seconds. The parameters of the fin equation \eqref{fineq} corresponding to this setting are $c_0 = 362,319,c_1 = 1, c_2 = 1$. The boundary conditions are chosen so that they agree with the data at the inner and outer radius as shown in Figure \ref{fig:bc}. 
\\ \indent
For the prior distribution we apply the same separable Gaussian process as in the simulation study \eqref{seperable}. We set the spatial length scale to $\rho_x=b-a=0.095$, as this is consistent with previous studies on the corresponding 1-dimensional stationary heat transfer problem described in \cite{biot}, and the temporal length scale is $\rho_t=400$. We apply the same network architecture as the simulation study, along with the adaptive training regime and delayed-acceptance HMC sampling scheme. This has similar timings to the simulation study detailed in Table \ref{table}, and after the warm-up period we obtained an ESS of 6,095 from 10,000 iterations. The inferred Biot number at various time points is shown in Figure \ref{real_dat0}. 
\begin{figure}[H]
\centerline{\includegraphics[width=3.in]{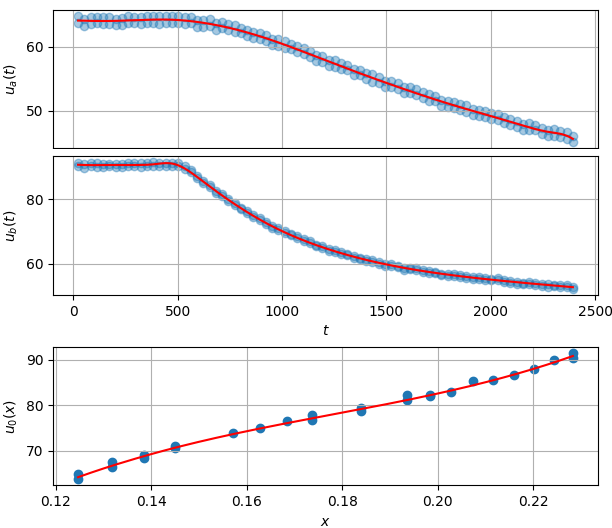}}	
\caption{Boundary and initial conditions for the experimental data computed using cubic spline regression.}
\label{fig:bc}
\end{figure}
\noindent
As expected our results show that there is very little change in $Bi(t,x)$ at the time points prior to the heater being turned off at 500 seconds, and after this we observe the transient behaviour of $Bi(t,x)$ as the disc cools. 
\begin{figure}[H]
\centerline{\includegraphics[width=4.5in]{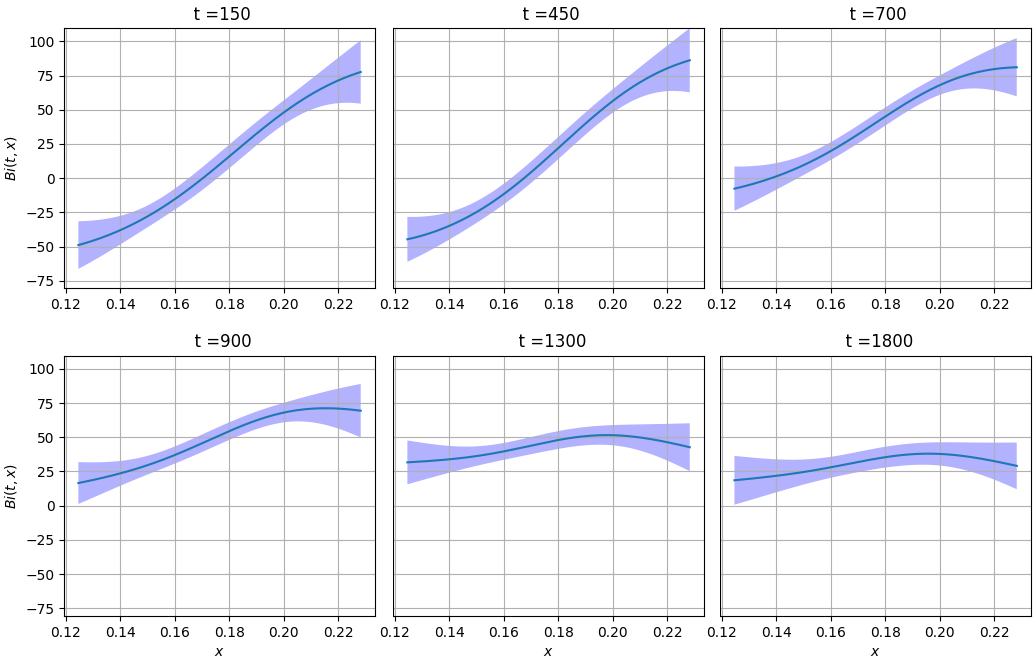}}	
\caption{The inferred Biot number and 95\% credible intervals at six time points.}
\label{real_dat0}
\end{figure}
\noindent
To validate the fit of the  inferred $Bi(t,x)$, we substitute it into the governing PDE and solve using a high resolution FD scheme to obtain the predicted temperature. Figure \ref{real_dat1} visualises these predicted temperatures and overlays the measured temperatures at various time points to demonstrate a good agreement between the data and PDE solution. 
\begin{figure}[H]
\centerline{\includegraphics[width=4.4in]{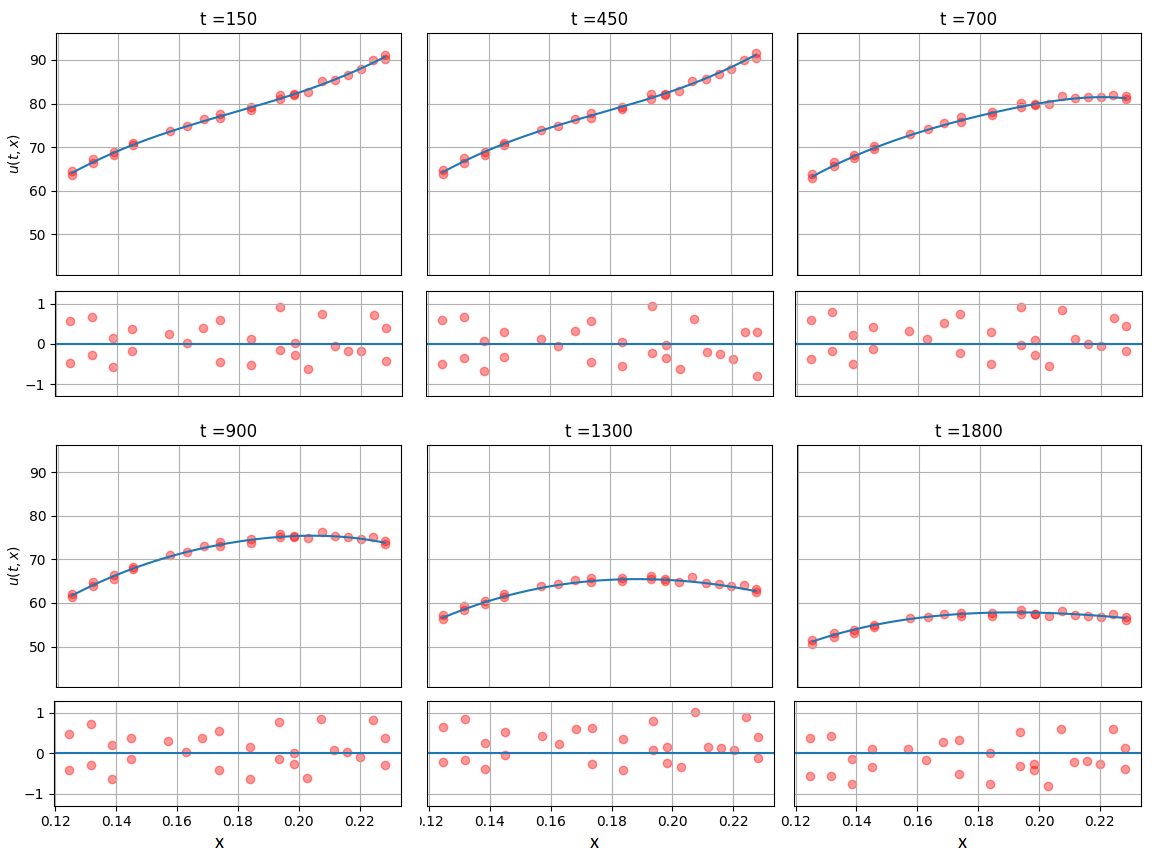}}	
\caption{Experimental data at various time points overlaid on the PDE solution using the inferred $Bi(t,x)$. The subplots below each panel are the residuals for the corresponding plot.}
\label{real_dat1}
\end{figure}
\noindent
A detailed physical interpretation of these results is outside the scope of this paper. Heuristically speaking, at $t=0$ the large positive $Bi(t,x)$ at the outer radius ($x=0.228$) represents the transfer of heat from the disc to the air at the outer radius where the disc is heated. This heated air then circulates and heats the disc at the cooler inner radius, thus producing the negative $Bi(t,x)$ that we observe in this region. When the heater is turned off at $t=500$ the disc cools and heat transfer slows, resulting in the decay in $Bi(t,x)$ shown. This demonstrates that the results of our analysis shown in Figure \ref{real_dat0} are physically interpretable, and as shown in Figure \ref{real_dat1} this result achieves a very good agreement with the experimental data when used to solve the PDE. In future this methodology will be applied to a range of more complex experimental settings, such as experiments with variable frequencies that mimic engine activity at various stages of flight.

\section{Closing remarks}
In this work we have developed a fully Bayesian approach to solving the PDE inverse problem for the unknown spatio-temporal Biot number. Our approach augments the general deep surrogate method for parametric PDEs by implementing a novel training scheme based on solving the PDE only over the approximated posterior distribution. To achieve this, we applied our deep learning approach to quickly obtain a Laplace approximation, and demonstrated that this gives reasonable estimates of the Biot number, but underestimates the overall uncertainty of this estimate. To sample from the posterior in an accurate and efficient manner, we applied a novel deep surrogate-based delayed-acceptance HMC scheme. This scheme utilises the fast evaluation and differentiation of the deep learning surrogate to make decorrelated proposals based on HMC trajectories and applies delayed-acceptance criteria to mathematically guarantee the accuracy of our posterior with respect to a finite-difference solver.  
\\ \indent
Since Biot number calculations are involved in the design of turbo-machinery, it is important that the results of our inferences have quantifiable accuracy. The delayed-acceptance step ensures sampling accuracy is consistent with a finite-difference solver, and our results demonstrate that the delayed-acceptance HMC scheme efficiently achieves a large effective sample size, thus ensuring a low statistical error. Our scheme achieves this accuracy while maintaining a lower cost per effective sample than Metropolis--Hastings schemes based on random walk or Langevin proposals, even if these schemes sacrifice the sampling accuracy of the delayed-acceptance step and rely solely on rapid surrogate evaluations. We also showed that the higher accuracy of our adaptively trained surrogate combined with HMC resulted in accurate posterior samples at a fraction of the cost. Our methodology is generalisable to other PDEs and applications, therefore in settings where approximate accuracy is sufficient the potential exists to apply this approach for substantial computational gains.
\bibliography{main}
\bibliographystyle{plain}
\newpage

\appendix
\section*{Appendix}

\section{Proof of convergence of surrogate posterior to the true posterior}
\label{A}
\subsection{Introduction} 
Consider the following second-order partial differential equation on a smooth bounded domain $D\subset \mathbb{R}^d$ for  $u\colon [0,T]\times D\to\mathbb{R}$,%
\begin{align*}
\mathcal{L}_{\text{Bi}} u :=  \frac{\partial u}{\partial t} - \nabla \cdot (A \nabla u) + c \cdot \nabla u - \text{Bi} u&=0,\qquad x\in D,\\
u(t,x)&=h(t,x),\qquad x\in\partial D, \numberthis  \label{pde} \\
u(0,x)&=u_0(x),\qquad x\in D,
\end{align*}
for $0\leq t \leq T$, given smooth initial data $u_0\colon D\to \mathbb{R}$ and boundary data $h\colon [0,T]\times\partial D \to \mathbb{R}$. The one-dimensional version of this equation ($D=(a,b)$) describes the experimental setting studied in the main article. The parameters are a positive-definite matrix-valued function $A\colon D\to \mathbb{R}^{d\times d}$, $c\colon D\to \mathbb{R}^d$, and $\text{Bi}\colon [0,T]\times D\to \mathbb{R}$. In turbomachinery applications~\cite{biot}, this PDE is known as the \emph{fin equation}, $u$ is temperature of a blade, and the Biot parameter $\text{Bi}$   measures the balance of conductive and convective heat transfer.  
 
 We assume knowledge of $c$, $A$, $h$, and $u_0$, and examine the inference of $\text{Bi}$ from observations of   $u$ at spatio-temporal locations $(\hat t_n,\hat x_n)\in(0,T]\times D$ for $n=1,\dots,N$. In the turbomachinery application, the measurements are taken by thermocouples that measure the temperature to a standard deviation $\sigma_\epsilon$.
We formulate the inference of $\text{Bi}$ as a Bayesian inverse problem.
Denote the solution of \cref{pde} for a given $\text{Bi}$ by $u(t,x, \text{Bi})$.
Define the mean-square error
  \begin{equation}
      \label{phi}\Phi(\theta) :=
\frac{1}{2 \sigma^2} \sum_{n = 1}^N | u (\hat t_n,\hat x_n, \text{Bi}) - \hat z_{n} |^2,\quad \theta=(\text{Bi},\sigma^2),
\end{equation}
where $\hat z_{n}$ are  measurements at location $\hat x_n$ and time $\hat t_n$.  Consider a prior distribution $\psi_0$ on $\theta=(\text{Bi},\sigma^2)\in C^{1,2}([0,T]\times \bar D)\times (0,\infty)$ (following standard notation for spaces of continuous functions; see \cref{not}).
We want to analyse the numerical approximation of the posterior distribution $\psi$ defined by
\[ d \psi (\theta) := \frac{1}{Z} \exp (- \Phi (\theta)) \,d \psi_0 (\theta) \]
for  normalisation constant $Z$, also known as the \emph{evidence}, given by \begin{equation}
Z := \int \exp (- \Phi (\theta)) \,d \psi_0 (\theta). \label{Z}
\end{equation}
In the main article, $\text{Bi}$ is presented as a parameterised function $\text{Bi}(\pmb \alpha)$ with parameter $\pmb \alpha$ and the prior specified on $\pmb \alpha$, where $\pmb \alpha$ gives the coefficients of the Chebyshev expansion of the function $\text{Bi}$. The following analysis is written for a prior specified directly on $\text{Bi}$; however, it can be extended to $\text{Bi}(\pmb\alpha)$, as long $\|\text{Bi}\|_\infty \leq C \|\pmb\alpha\|_\infty$ (for some constant $C$).

\subsubsection{Notation}\label{not}
For metric space $X$ and Hilbert space $Y$, denote by $L^2(X,Y)$ (respectively, $L^2(X)$)  the Hilbert space of square-integrable functions from $X\to Y$ (resp., from $X\to \mathbb{R}$). $H^{1/2}(\partial D)$ is the Sobolev space of functions from the boundary $\partial D$ of $D$ to $\mathbb{R}$ with one-half weak derivatives; its norm can be characterised by
\[
\norm{u}_{H^{1/2}(\partial D)}=\inf_{\gamma v=u; v\in H^1(D)} \norm{v}_{H^1(D)},\qquad \forall u\in H^{1/2}(\partial D),%
\]
where $H^1(D)$ is the classical Sobolev space with one square-integrable weak derivative and $\gamma$ is the trace operator that restricts functions on $D$ to the boundary $\partial D$. See~\cite{Renardy2006-dt,Leoni2017-fq}.

 For metric spaces $X,X_1,X_2$ and a Banach space $Y$,  $C(X)$, $C(X,Y)$, $C^{k,\ell}(X_1\times X_2,Y)$ are the usual spaces of continuous functions $X\to \mathbb{R}$, $X\to Y$ and $(k,\ell)$-times differentiable functions $X_1\times X_2 \to Y$. If $X$ is compact, $C(X)$ and $C(X, Y)$ are  Banach spaces with the  supremum norm $\norm{\cdot}_\infty$. The closure of the domain $D$ is denoted $\bar D$ and we frequently consider $C^{1,2}([0,T]\times \bar D)$ and $C(\bar D,Y)$.

We make the following assumptions on the coefficients in \cref{pde}.
\begin{assumption}\label[assumption]{assc}
The boundary data $h\in C([0,T],H^{1/2}(\partial D))$, initial data $u_0\in L^2(D)$, and  field $c\in C(\bar D, \mathbb{R}^d)$. The diffusion $A\in C(\bar D,\mathbb{R}^{d\times d})$ and is uniformly positive-definite: 
\[
\inf_{{x\in\bar D}\vphantom{ux\in\mathbb{R}^d}} \inf_{{u\in\mathbb{R}^d}\vphantom{xb,\in\bar D\mathbb{R}^d}}
\frac{\langle u,A(x)u\rangle} {\norm{u}^2}>0.
\]
\end{assumption}
Under this assumption, \cref{pde} has a well-defined solution in $L^2(0,T,H^1(D))$ (see \cref{pde_reg}).

\subsubsection{Neural network approximation}
To solve the Bayesian inverse problem, we introduce a surrogate $\hat{u} \approx u$ that captures variation of the solution in time $t$, space $x$, and the parameter $\text{Bi}$. The surrogate is defined by a parameterised deep Galerkin method~\cite{sir17}, which  means $\hat{u}(t,x,\text{Bi})$ is a deep neural network (DNN) that is trained to optimise a physics-informed loss function. In the main article, we represent $\text{Bi} \in C([0,T]\times D)$ as a truncated Chebyshev expansion, whose coefficients become inputs to the surrogate $\hat u$.
 
 We describe the loss function now and  derive it later. For a given $\mathcal{L}_\text{Bi}$, $\text{Bi}$, $u_0$ and $h$ in \cref{pde},  define $F$ for $ w\in C^{1,2}( [0,T]\times \bar D)$ by
\begin{gather}\begin{split}
F(w )&:=\|h-\gamma\, w \|^2_{L^2((0,T),H^{1/2}(\partial D))}\\&\quad+  \| u_0- w (0,\cdot)\|^2_{L^2(D)} 
  +   \| \mathcal L_{\text{Bi}} w \|^2_{L^2((0,T),L^2(D))},
\end{split}\label{Fd}\end{gather}
(where $\gamma$ is the trace operator);
 for $d=1$ with $D=(a,b)$, $F$ may be replaced by
\begin{gather}
    \begin{split}
F(w )&:=\|h(\cdot,a)-w (\cdot,a)\|^2_{L^2(0,T)}+\|h(\cdot,b)-w(\cdot,b)\|^2_{L^2(0,T)}\\&\quad+  \| u_0- w (0,\cdot)\|^2_{L^2(a,b)} 
+   \| \mathcal L_{\text{Bi}} w \|^2_{L^2((0,T),L^2(a,b))}.
\end{split}\label{F1}
\end{gather} 
The last definition agrees with the $F$ defined in the main article.

We assume a distribution $\psi_{\tmop{approx}}$ is available  that is close to $\psi$ in the sense that the Radon--Nikodym derivative $d\psi_{\tmop{approx}}/d\psi$ is well-defined and close to one. We suppose that  $\hat{u} $ has been trained by optimising the loss function
\[
 \int  q_\sigma(F(\hat{u} (\cdot,\text{Bi})))\, d \psi_{\tmop{approx}} (\theta), \qquad \theta=(\text{Bi},\sigma^2),
 \]
where $q_\sigma\colon \mathbb{R}^+\to \mathbb{R}^+$ is some  increasing function that we make precise later (see \cref{main,main2}).

Given the surrogate $\hat{u}$, we define an approximate posterior distribution
$\hat{\psi}$ by \[ d \hat{\psi} (\theta) := \frac{1}{\hat{Z}} \exp (- \hat{\Phi} (\theta)) \,d \psi_0 (\theta), \]
where $\hat{Z}$ and $\hat{\Phi}$ are defined similarly to $Z$ and $\Phi$ (in \cref{phi,Z}). 
We're interested in the quantifying the approximation of $\psi$ by $\hat{\psi}$ in terms of the size of the loss function. 

\subsubsection{Main result}

We state the two main theorems relating the loss function and the quality of the surrogate to the posterior error. These show that the error (in the Hellinger distance) in approximating the posterior distribution can be bounded in terms of the loss function. 

First, we reformulate the log-likelihood function $\Phi$ for the data in terms of a projection operator$\ \pi$. 
\begin{assumption}\label[assumption]{assb}
The projection $\pi\colon L^2((0,T), H^1(D)) \to \mathcal{P}_m$, the \\ piecewise polynomials in $(t,x)$ of degree $m$ or less, satisfies
\begin{equation}
\norm{\pi w-w}_{L^2((0,T)\times D)}\leq C \norm{w}_{L^2((0,T),H^1(D))}, \qquad \forall w\in L^2((0,T),H^1(D)),
\label{bh}
\end{equation}
for a constant $C$. 
\end{assumption}
The Bramble--Hilbert lemma gives examples of such projections $\pi$ defined by averaged Taylor expansions (including the case of no derivatives~\cite{Brenner2008-jg}). We assume that  $\Phi$ can be written  \[\Phi(\theta)=\frac{1}{2\sigma^2}\norm{\pi u(\cdot,\cdot,\text{Bi})- z}^2_{L^2((0,T)\times D)},
\]
for some $z\in L^2((0,T)\times D)$ and $\sigma>0$.  The $\Phi$ defined in \cref{phi} can be expressed this way via piecewise-constant functions. 

Next we state some moment assumptions on the solution to \cref{pde}.
\begin{assumption}\label[assumption]{assa} 
Suppose that, for any $p>1$,
\[
\int
 e^{p\norm{\text{Bi}}_\infty\,T}\,d\psi_{0,1}(\text{Bi}) <\infty,
\]
where $\psi_{0,1}$ denotes the $\text{Bi}$-marginal distribution of the prior distribution  $\psi_0$.
\end{assumption}
This assumption holds if $\psi_{0,1}$ is a Gaussian random field with a smooth mean and covariance function (see~\cite[Chapter 7]{Lord2014-co}).
  
We use the Hellinger distance $d_{\tmop{Hell}}$ to measure the approximation error in $\hat{\psi}$.
For  two measures $\mu$ and $\nu$ with
Radon--Nikodym derivatives $d \mu / d \mu = f$ and $d \nu / d \mu = g$, the
Hellinger distance $d_{\tmop{Hell}}$ (see~\cite{stuart}) is defined by
\begin{equation}
  d_{\tmop{Hell}} (\mu, \nu)^2 := \frac{1}{2} \int \left( \sqrt{f(\theta)} -
  \sqrt{g(\theta)}  \right)^2 d \mu (\theta).
\end{equation}
\begin{theorem}\label{main}
Let \cref{assa,assb,assc} hold.
  For any distribution $\psi_\text{approx}$ that is absolutely continuous
  with respect to $\psi$, there exists a constant $C$ such that
  \begin{align*}
    d_{\tmop{Hell}} (\psi, \hat{\psi})^2 & \leq  
    C\left\| \frac{d \psi}{d \psi_{\tmop{approx}}} \right\|_\infty^{1/2} \left(\int \frac{1}{\sigma^4}  F(\hat{u} (,\cdot,\text{Bi}))^2\, d \psi_{\tmop{approx}} (\text{Bi},\sigma^2)\right)^{1/2} \\
    &\qquad+ \frac{C}{Z}
 \left(\int \frac{1}{\sigma^8}  F(\hat{u} (\cdot,\text{Bi}))^{6}\, d \psi_0 (\text{Bi},\sigma^2)\right)^{1/2}.
  \end{align*}
\end{theorem}

The size of the evidence $Z$ is highly significant. For good choices of prior and well-trained networks $\hat{u} $, $Z$ is large and $F(\hat{u} )$ is small, so that the second-term is negligible. The error is dominated by the second moment of $F$ with respect to the initial distribution $\psi_{\tmop{approx}}$.

A poor prior distribution leads to small values for $Z$ and the second term will be very large and approximation more challenging. 

The next theorem assumes all samples from the prior give a Biot number $\text{Bi}$ that is uniformly bounded.

\begin{theorem}\label{main2} In addition to the assumptions of \cref{main}, suppose that $\Phi(\theta)$ is  almost surely bounded for $\theta\sim \psi_{0}$. Then,
\begin{gather}
\begin{split}
    d_{\tmop{Hell}} (\psi, \hat{\psi})^2 & \leq  
    C\left\| \frac{d \psi}{d \psi_{\tmop{approx}}} \right\|_\infty\int\frac{1}{\sigma^2}  F(\hat{u} (\cdot,\text{Bi}))\, d \psi_{\tmop{approx}} (\text{Bi},\sigma^2) \\
    &\qquad+C    \left\| \frac{d \psi}{d \psi_{\tmop{approx}}} \right\|_\infty
 \int\frac{1}{\sigma^6}  F(\hat{u} (\cdot,\text{Bi}))^3\, d \psi_{\tmop{approx}}(\text{Bi},\sigma^2).  
\end{split}   \label{t15}
\end{gather}
  \end{theorem}
  The surrogate does not know about $\sigma$, which accounts for observation error, and $\sigma$ must be integrated out to see the relationship with the surrogate loss:
  \begin{align*}
  \int \frac{1}{\sigma^2}  F(\hat{u} (,\cdot,\text{Bi}))\, d \psi_{\tmop{approx}} (\text{Bi},\sigma^2) 
  = \int F(\hat u(\cdot,\text{Bi})) \,d\pi(\text{Bi})=\text{surrogate loss}, 
  \end{align*}
  where $d\pi(\text{Bi}):= \int_{\mathbb{R}} \frac{1}{\sigma^2} d \psi_{\tmop{approx}} (\text{Bi},\sigma^2)$.
  The main article trains the surrogate with respect to the right-hand side (identifying $\pmb \alpha$ with $\text{Bi}$ and $\pi$ with $\pi^{\alpha}$).  

There is no dependence on the evidence  $Z$ and  the leading-order term is the first moment of $F$ (and not the second moment found in \cref{main}). 

Notably, the coefficients in the right-hand side of \eqref{t15} contain the supremum norm of Radon--Nikodym derivatives. These terms attain their minimum of one precisely when $\psi_{\tmop{approx}} = \psi$; thus, for a fixed loss-function value, this bound is minimised when the training measure is set to the true posterior. Of course, in practice, the true posterior is unavailable, and so $\psi_{\tmop{approx}}$ is set to the best known approximation of the posterior.

\subsection{Preparatory theory}

We collect some theory that will be used in proving the above results. Throughout $C$ is used to denote a generic constant (i.e., it can be chosen independently of $\theta=(\text{Bi},\sigma^2)$) that varies from place to place. We write $C(\text{Bi})$ for a generic constant that depends on the Biot number.

\subsubsection{Estimates for PDE}
We describe how the solution $u$  is approximated by a second function $\hat{u} $ in terms of $F(\hat{u} )$ defined in \cref{F1,Fd}. We first state a lemma concerning the boundary condition.
\begin{lemma}[boundary extension operator]\label[lemma]{bdry}%
Let $D$ be a bounded uniformly Lipschitz domain. If $h\in H^{1/2}(\partial D)$, then there exists $e\in H^1(D)$ such that $\norm{e}_{H^1(D)} \leq C \norm{h}_{H^{1/2}(\partial D)}$ and $\norm{e}_{L^2(D)}\leq C\norm{h}_{L^2(\partial D)}$ for a constant $C$ that depends only on the domain $D$.
\end{lemma}

\begin{proof}
  See~\cite[Theorem 18.40]{Leoni2017-fq}.
\end{proof}
\begin{theorem}\label{pde_reg} Let \cref{assa,assc} hold.  For each $\text{Bi} \in C([0,T]\times \bar D)$, \cref{pde} has a well-defined weak solution $u(\cdot,\text{Bi})\in L^2((0,T),H^1(D))$.  

Let $\hat u \in C^{1,2}([0,T]\times \bar D)$. There exists  a constant $C(\text{Bi})$ such that
   \begin{align*}
  \|u(\cdot,\text{Bi})-\hat{u}(\cdot,\text{Bi}) \|^2_{L^2(0,T, H^1(D))}&\leq C(\text{Bi})\,F(\hat{u} ),
  \end{align*}  
  for $F$ defined in \cref{Fd,F1}. In particular, note that
   \begin{align*}
  \|u(\cdot,\text{Bi}) \|^2_{L^2((0,T), H^1(D))}&\leq C(\text{Bi}) F(0).
  \end{align*}  
  The function $C(\text{Bi})$ can be bounded by $C\, e^{\norm{ \text{Bi}}_\infty  T}$.
\end{theorem}
\begin{proof} Let $g=\mathcal L_{\text{Bi}} \hat{u} $. Consider $e(t)\in H^1(D)$ such that $e=\hat{u} -u$ on $[0,T]\times \partial D$ (as provided by \cref{bdry}).
Then $\epsilon:=  u-\hat{u} -e$ satisfies $\mathcal{L}_{\text{Bi}} \epsilon =-g-\mathcal{L}_{\text{Bi}} e $ with $\epsilon (t,x)=0$  for $x\in \partial D$. 
 Hence, Example 11.5 and (11.15) from~\cite{Renardy2006-dt} apply and we find that
  \begin{align*}
  \|\epsilon\|_{L^2((0,T), H^1(D))}\leq C(\text{Bi}) &\| (u-\hat{u} -e)(0,\cdot)\|_{L^2(D)} \\
  &+ C(\text{Bi}) \|g+\mathcal{L}_{\text{Bi}} e\|_{L^2((0,T),H^{-1}(D))}.
  \end{align*}
  The constant $C(\text{Bi})$ depends on the size of $u$ and can be bounded by $C\, e^{\norm{\text{Bi}}_\infty\, T}$.
  
  On the right-hand side, the first term is the error in the initial data and the second term is error on the interior.   Note that $\norm{(u-\hat{u} -e)(0,\cdot)}_{L^2(D)} \leq C(\text{Bi})\norm{(u-\hat{u})(0,\cdot) }_{L^2(D)}$ following \cref{bdry}.
  As $g+\mathcal{L}_{\text{Bi}} e=-\mathcal{L}_{\text{Bi}} \epsilon=\mathcal{L}_{\text{Bi}} \hat{u}  +\mathcal{L}_{\text{Bi}} e$
  and $u-\hat{u} = e+\epsilon$, we have
     \begin{align*}
&  \|u-\hat{u} \|_{L^2((0,T), H^1(D))}\leq \|e\|_{L^2((0,T), H^1(D))}+ C(\text{Bi})\norm{ (u-\hat{u} )(0,\cdot)}_{L^2(D)}+\\\qquad&\qquad+C(\text{Bi}) \| \mathcal L_{\text{Bi}} \hat{u} \|_{L^2((0,T),H^{-1}(D))}+ C(\text{Bi})\|\mathcal{L}_{\text{Bi}} e\|_{L^2((0,T),H^{-1}(D))}.
  \end{align*}
  Now, \begin{align*}
  \norm{\mathcal{L}_{\text{Bi}} e}_{H^{-1}(D)}
  &\leq C(1+\norm{\text{Bi}}_\infty)\norm{e}_{H^1(D)}\\ &\leq C(1+\norm{\text{Bi}}_\infty)\norm{(h-\hat{u} )}_{H^{1/2}(\partial D)}.\end{align*}
Finally then,
   \begin{align*}
  \|u-\hat{u} \|_{L^2((0,T), H^1(D))}&\leq C(\text{Bi}) \Big[\|h-\hat{u} \|_{L^2((0,T), H^{1/2}(\partial D))}\\&\quad+  \| u_0- \hat{u} (0,\cdot)\|_{L^2(D)}+  \\&
  \quad+   \| \mathcal L_{\text{Bi}} \hat{u} \|_{L^2(0,T;H^{-1}(D))}\Big].
  \end{align*}
  Here the constant $C(\text{Bi})$ depends on $u$ and can be bounded by $C_0 e^{\bar{\text{Bi}}\, T}$.
  
  In the one-dimensional case, $D=(a,b)$, $e$ can be taken to be a linear function function.
     \begin{align*}
  \|u-\hat{u} &\|_{L^2((0,T), H^1(D))} \\
  &\leq    C(\text{Bi})\Big[\|h(\cdot,a)-\hat{u} (\cdot,a))\|_{L^2(0,T)} + \| h(\cdot,b)-\hat{u} (\cdot,b)\|_{L^2(0,T)} \\&\qquad
+   \| u_0- \hat{u} (0,\cdot)\|_{L^2(a,b)}+   \| \mathcal L_{\text{Bi}} \hat{u} \|_{L^2((0,T)H^{-1}(a,b))}\Big].
  \end{align*}

\end{proof}


%

\subsubsection{Hellinger distance}



The posterior-distribution error measured by the Hellinger distance is
\[ d_{\tmop{Hell}} (\psi, \hat{\psi})^2 = \frac{1}{2}
   \int \left( \frac{1}{\sqrt{Z} } \exp \left( - \frac{1}{2} \Phi (\theta) \right)
   - \frac{1}{\sqrt{ \hat{Z}}} \exp \left( - \frac{1}{2} \hat{\Phi} (\theta) \right)
   \right)^2 d \psi_0 (\theta) . \]
\begin{lemma}\label[lemma]{hell_bd}
   For a constant $C
  > 0$,
  \begin{gather}\begin{split}
   d_{\tmop{Hell}} (\psi, \hat{\psi})^2 &\leq
     C \int (\Phi (\theta) - \hat{\Phi} (\theta))^2 \,d \psi  (\theta) \\&\quad+
     \frac{C}{Z} \int (\Phi (\theta) - \hat{\Phi} (\theta))^4\, d \psi_0  (\theta) . 
\end{split}\label[ineq]{ne}\end{gather}

Suppose that $\Phi(\theta)$ is almost-surely bounded if $\theta\sim  \psi_0$. Then,
 \begin{gather}\begin{split}
   d_{\tmop{Hell}} (\psi, \hat{\psi})^2 &\leq
     C \int (\Phi (\theta) - \hat{\Phi} (\theta))^2 \,d \psi  (\theta) \\&\quad+
     C \int (\Phi (\theta) - \hat{\Phi} (\theta))^4\, d \psi  (\theta) . 
\end{split}\label[ineq]{wo}\end{gather}
\end{lemma}

\begin{proof}
  First write as in~\cite[Theorem 4.6]{stuart},
  \[ 2 d_{\tmop{Hell}} (\psi, \hat{\psi})^2 = I_1 +
     I_{2,} \]
  where
  \begin{align*} I_1 &= \frac{2}{Z} \int \left( \exp \left( - \frac{1}{2} \Phi (\theta) \right)
     - \exp \left( - \frac{1}{2} \hat{\Phi} (\theta) \right) \right)^2 \,d \psi_0 (\theta), \\ I_2& = 2 \left| Z^{- 1 / 2} - \hat{Z}^{- 1 / 2}  \right|^2 \int^{} \exp
     \left( {- \hat{\Phi}}  (\theta) \right) \,d \psi_0 (\theta) . \end{align*}
  Consider first $I_{1}$: by definition of $\psi$,
  \[ I_1 = 2 \int \left( \left( 1 - \exp \left( \frac{1}{2} (\Phi
     (\theta) - \hat{\Phi} (\theta)) \right) \right)^2 \,d \psi (\theta) . \right. \]
  To estimate the bracket, recall that, for any $x \in \mathbb{R,}$
  \[ 1 - \exp (x) = - x - \frac{1}{2} \exp (\xi) x^2, \quad \text{for some $\xi$
     between $0$ and $x$}, and \]
  \[ (1 - \exp (x))^2 \leq x^2, \quad \text{for $x \leq 0$}. \]
  Thus, using $(a + b)^2 \leq 2 a^2 + 2 b^2$,
  \begin{align*} I_1 &\leq 4 \int \frac{1}{4} (\Phi (\theta) - \hat{\Phi} (\theta))^2 \,d
     \psi  (\theta)\\&\quad + \frac{1}{4 } \int \exp (\Phi (\theta) - \hat{\Phi} (\theta))
     (\Phi (\theta) - \hat{\Phi} (\theta))^4 \,d \psi  (\theta) . \end{align*}
If $\Phi$ is uniformly bounded, it is easily seen that $I_1$ satisfies a bound like \ref{wo}. 
In general, the second integral is bounded by
  \[\frac{1}{Z}
  \int
     (\Phi (\theta) - \hat{\Phi} (\theta))^4\, d \psi_0  (\theta).
  \]
  Then,
  \[ I_1 \leq 4 \int \frac{1}{4} (\Phi (\theta) - \hat{\Phi} (\theta))^2 \,d
     \psi  (\theta) + \frac{C}{4 Z} \Big(\int (\Phi (\theta) - \hat{\Phi} (\theta))^4 \,d
     \psi (\theta) \Big). \]
     This leads to a bound on $I_1$ like \ref{ne}.
     
  For $I_2$, 
       as $(Z^{-1/2}-\hat Z^{-1/2}) = (Z-\hat Z)/(Z^{-1/2}+\hat{Z}^{-1/2})$, 
     
       \[ I_2 \leq 2 \left| Z - \hat{Z} \right|^2 \frac{1}{Z} \int^{} \exp
     \left( {- \hat{\Phi}}  (\theta) \right) \,d \psi_0 (\theta). \]
    Using Jensen's inequality,
  \[ | Z - \hat{Z} |^2 \leq \int (\exp (- \Phi (\theta)) - \exp (- \hat{\Phi} (\theta)))^{^2 }\, d
     \psi_0 (\theta). \]
  Hence, as $\hat{\Phi}\geq 0$,
  \[ I_2 \leq 2 \frac{1}{Z} \int (\exp (- \Phi (\theta)) - \exp (- \hat{\Phi} (\theta)))^{^2 } \,d
     \psi_0 (\theta). \]
 This has the same form as $I_1$, which means we can bound $I_2$ in the same way that we bound $I_1$, to complete the proof.
\end{proof}

\subsection{Proof of main theorem}

Our main results describes the the approximation of the posterior distribution in the Hellinger distance in terms of $F$. We first estimate the difference between $\Phi$ and $\tilde \Phi$ before proving \cref{main,main2}.

\begin{lemma}\label[lemma]{lhood-bd} 
For $\theta= (\text{Bi},\sigma^2)$, 
\begin{align*}
     |\Phi(\theta) - \hat{\Phi}(\theta)|
      &\leq   \frac{1}{2\sigma^2}\,C(\text{Bi})\,\norm{\pi u(\cdot,\text{Bi})-\pi \hat{u}(\cdot,\text{Bi}) }_{L^2((0,T)\times D)},  
\end{align*}
with $C(\text{Bi})=\norm{\pi u(\cdot,\text{Bi})-\pi \hat{u}(\cdot,\text{Bi})}_{L^2((0,T)\times D)}+ 2 \norm{\pi u(\cdot, \text{Bi})-z}_{L^2((0,T)\times D)}$.
\end{lemma}

\begin{proof} Assuming without loss of generality that $\sigma=1$, we recall that $\Phi  =\norm{\pi u -z}^2_{L^2((0,T)\times D)}$
  (dropping explicit dependence on $\theta=(\text{Bi},\sigma^2)$). Hence,
  \[ \Phi  - \hat{\Phi} 
   =  \langle \pi u- \pi \hat{u} , \pi u+ \pi \hat{u} - 2z \rangle_{L^2((0,T)\times D)}, \]
  using $(x - y)^2 - (z - y)^2 = (x - z) (x + z - 2 y)$. Let $\delta=\pi u- \pi \hat u$, so that    \begin{align*}
|  \Phi  - \hat{\Phi}|  &= | \langle \delta, 2(\pi u-z)-\delta \rangle_{L^2((0,T)\times D)}|  \\
&\leq \norm{\delta}^2_{L^2((0,T)\times D)}  + 2\norm{\delta}_{L^2((0,T)\times D)}  \norm{\pi u- z}_{L^2((0,T)\times D)}
  \end{align*}
  by the Cauchy--Schwarz inequality, which completes the proof.
\end{proof}

\begin{corollary}\label[corollary]{cor-blimey}
Under the assumption of \cref{lhood-bd} and \cref{pde_reg}, for $\theta=(\text{Bi},\sigma^2)$,
\begin{align*}
   |\Phi(\theta)  - \hat{\Phi}(\theta) |^2 %
   &\leq C(\text{Bi})\,\frac{1}{\sigma^2} F(\hat{u}(\cdot,\text{Bi}) ),
     \end{align*}
     for
     $C(\text{Bi})=C\,(F(0) + F(\hat u(\cdot,\text{Bi})))^{1/2}\, e^{2\norm{\text{Bi}}_\infty\,T}$ for a constant $C$.
\end{corollary}
\begin{proof} 
For $\delta=\pi u-\pi \hat{v} $, note $\norm{\delta}_{L^2([0,T]\times D)}\leq C(\text{Bi}) \norm{u-\hat{u} }_{L^2((0,T),H^1(D))}$ with $C(\text{Bi})\leq C_0 e^{\bar{\text{Bi}}\,T}$ by \cref{bh}.
Applying \cref{lhood-bd} with \cref{pde_reg},
\begin{align*}
C(\text{Bi})&\leq C(\norm{u(\cdot,\text{Bi})-\hat{u}(\cdot,\text{Bi})}_{L^2((0,T)\times D)} + \norm{u(\cdot,\text{Bi})}_{L^2((0,T)\times D)})\\
&\leq Ce^{\norm{\text{Bi}}_\infty\, T} (F(0) + F(\hat u))^{1/2}.
\end{align*}
The proof is complete.
\end{proof}

\begin{proof}[Proof of \cref{main}] 
  From \cref{cor-blimey}, for $p\geq 1$
,  \[ \int (\Phi (\theta) - \hat{\Phi} (\theta))^{2p} \,d \psi  (\theta)
 \leq \int \frac{1}{\sigma^{2p}} C(\text{Bi})^p\,F(\hat u(\cdot, \text{Bi}))^{p}\, d \psi (\theta). 
\]
Hence, using the Cauchy--Schwarz inequality and assuming boundedness of moments as given by \cref{assa}, we  replace $C(\text{Bi})$ by a constant in 
  \[ \int (\Phi (\theta) - \hat{\Phi} (\theta))^{2p}\, d \psi  (\theta)
   \leq C\left(\int\frac{1}{\sigma^{4p}} (F(\hat{u} ) ^{2p}+F(\hat{u})^{3p} )\, d \psi (\theta)\right)^{1/2}. \]
\cref{hell_bd} gives that
  \begin{align*}
   d_{\tmop{Hell}} (\psi, \hat{\psi})^2 &\leq
     C \int (\Phi (\theta) - \hat{\Phi} (\theta))^2\, d \psi  (\theta) \\
     &\quad+
     \frac{C}{Z} 
     \int (\Phi (\theta) - \hat{\Phi} (\theta))^4\, d \psi_0  (\theta). 
\end{align*}
Hence,
\begin{align*}
   d_{\tmop{Hell}} (\psi, \hat{\psi}) ^2 &\leq 
       C\left(\int \frac{1}{\sigma^4} 
       F(\hat u(\cdot,\text{Bi}))^2 \,d \psi (\theta) \right)^{1/2} \\
       &\quad+
     \frac{C}{Z}
      \Big(\int \frac{C}{\sigma^{8}}
       F(\hat{u} (\cdot;\text{Bi}))^{6} \,d \psi_0  (\theta)
       \Big)^{1/2}. 
\end{align*}
    Finally, suppose that $\frac{d \psi}{d \psi_{\tmop{approx}}}$ is 
  bounded. Then,
  \begin{align*} \int (\Phi (\theta) - \hat{\Phi} (\theta))^2 \,d \psi  (\theta)& \leq C \left\|
     \frac{d \psi}{d \psi_{\tmop{approx}}} \right\|_\infty^{1/2} \left(\int\frac{1}{\sigma^4} F(\hat{u} (\cdot,\text{Bi}))^2 \,d\psi_{\tmop{approx}} (\theta)\right)^{1/2} \\&+
     \frac{C}{Z}  \left(\int  \frac{1}{\sigma^{8}} F(\hat{u} (\cdot,\text{Bi}))^{6} \,d\psi_0 (\theta)\right)^{1/2}.
     \end{align*}
  for an enlarged constant $C$ as required.
\end{proof}

\begin{proof}[Proof of \cref{main2}.]
  The argument is simplified under the boundedness condition on the prior $\psi_0$, as the Cauchy--Schwarz step is not necessary and it holds that,
  for $p\geq 1$
,  \[ \int (\Phi (\theta) - \hat{\Phi} (\psi))^{2p} \,d \psi  (\theta)
 \leq \int\frac{1}{\sigma^{2p}} \,F(\hat u(\cdot;\text{Bi}))^{2p}+F(\hat u(\cdot;\text{Bi}))^{3p}\, d \psi (\theta). 
\]
With boundedness,  \cref{hell_bd} says that
  \begin{align*}
   d_{\tmop{Hell}} (\psi, \hat{\psi})^2 &\leq
     C \int (\Phi (\theta) - \hat{\Phi} (\theta))^2 \,d \psi  (\theta) \\&\quad+
     C \int (\Phi (\theta) - \hat{\Phi} (\theta))^4 \,d \psi_0  (\theta) 
\end{align*}
with no dependence on the normalisation constant $Z$. Putting everything together,
  \begin{align*} \int (\Phi (\theta) - \hat{\Phi} (\theta))^2 d \psi  (\theta)& \leq C \left\|
     \frac{d \psi}{d \psi_{\tmop{approx}}} \right\|_\infty \int \frac{1}{\sigma^2}F(\hat{u} (\cdot,\text{Bi})) \,d\psi_{\tmop{approx}} (\theta)\\&\qquad +C   \left\|
     \frac{d \psi}{d \psi_{\tmop{approx}}} \right\|_\infty \int\frac{1}{\sigma^6} F(\hat{u} (\cdot,\text{Bi}))^3 \,d\psi_{\tmop{approx}} (\theta). \\
     \end{align*}
\end{proof}

\section{MAP estimation using local deep surrogates}
\label{B}
To train the Laplace-based surrogate which is used to begin sampling we must first compute the MAP estimate, that is we are interested in finding
\begin{align}
    (\boldsymbol{\alpha}^*,\sigma^*) = \text{argmax}\left\{\text{log}\ p(\boldsymbol{\alpha},\sigma|\hat{t},\hat{x},\hat{z})\right\}. \label{map2}
\end{align}
Starting with an initial guess $(\boldsymbol{\alpha}_0,\sigma_0)$, we solve this optimisation problem in an approach that is similar in concept to trust region optimisation. More specifically, we train a local surrogate $\hat{u}(t,x,\boldsymbol{\alpha})$ at $\boldsymbol{\alpha}_0$ by minimising the loss function with $\pi^{\alpha}=\pi^{\alpha_0}$, where $\pi^{\alpha_0}$ is a probability measure with mass focused locally around $\boldsymbol{\alpha}_0$. When the local surrogate is accurate, as judged by achieving a sufficiently low loss function value, we substitute this into the log-posterior \eqref{map2} and take a gradient ascent step with respect to $(\boldsymbol{\alpha}_0,\sigma_0)$ to maximise the resulting expression. This process repeats in an alternating manner, whereby at iteration $n$ the network is trained over $\pi^{\alpha_n}$ to ensure local accuracy at $\boldsymbol{\alpha}_n$ and another gradient ascent step is performed to maximise the posterior. Upon convergence the MAP estimate is the final value of $\boldsymbol{\alpha}_N$, and the Hessian matrix used in the Laplace approximation covariance is constructed by automatic differentiation of the log-posterior at $\boldsymbol{\alpha}_N$. This Laplace approximation is then used as the training measure $\pi^\alpha$ in the loss function in order to train the Laplace-based surrogate subsequently used to begin MCMC sampling.
\\ \\
In our implementation we choose $\pi^{\alpha_n} \sim MVN(\boldsymbol{\alpha}_n,\Lambda^n)$ for the local measure, where $\Lambda$ is a diagonal matrix. For the variance of this distribution we choose $\Lambda^n_{i,i}=\lambda_n/2^{2k}$, where $k$ is the degree of the polynomial term corresponding to the $i^{th}$ position. The decay in the variance of $\pi^{\alpha_n}$ is manually set, starting with $\lambda_0=20$ and ending with $\lambda_{\text{end}}=0.5$. This local surrogate is significantly cheaper to compute than the general surrogate, taking only a few seconds to reduce the loss function to below the level achieved by training a general surrogate for 15 hours, and ultimately achieving a significantly higher accuracy in a fraction of the time.

\subsection*{Acknowledgements}
We thank our colleague Hui Tang  from the Department of Mechanical Engineering at the University of Bath for their helpful discussions regarding the context of this work.
\\ \\
Teo Deveney  is supported by a scholarship from the EPSRC Centre for Doctoral Training in Statistical
Applied Mathematics at Bath (SAMBa), under the project EP/L015684/1.

\end{document}